\documentclass[11pt]{article}
\usepackage{color}
\usepackage{array}
\usepackage{algpseudocode}
\usepackage{algorithm}
\usepackage{amssymb}   
\usepackage{amsthm}    
\usepackage{amsmath}   
\usepackage{stmaryrd}  
\usepackage{titletoc}  
\usepackage{mathrsfs}  
\usepackage{graphicx}
\usepackage{hyperref}
\newlength{\defbaselineskip}
\newcommand{\setlinespacing}[1]%
           {\setlength{\baselineskip}{#1 \defbaselineskip}}

\newtheoremstyle{myexample} 
  {3pt}    
  {3pt}    
  {}       
  {}       
  {\bfseries} 
  {.}      
  { }      
  {}       

\theoremstyle{myexample}
\newtheorem{example}{Example}
\theoremstyle{plain}
\newtheorem{thm}{Theorem}[section]

\newtheorem{lem}[thm]{Lemma}

\theoremstyle{definition}

\newtheorem{rmk}{Remark}[section]
\newcommand{\RN}[1]{%
  \textup{\uppercase\expandafter{\romannumeral#1}}%
}

\newcommand{\bC}{\mathbb{C}}
\newcommand{\bE}{\mathbb{E}}

\newcommand{\bP}{\mathbb{P}}
\newcommand{\bR}{\mathbb{R}}
\newcommand{\bN}{\mathbb{N}}

\newcommand{\sF}{\mathscr{F}}

\newcommand{\sech}{\text{sech}}
\usepackage{float}
\makeatletter\@addtoreset{equation}{section} \makeatother
 \allowdisplaybreaks
\begin{document}

\title{Feynman-Kac Formula for Time-Dependent Nonlinear Schr\"odinger  Equations with Applications in Numerical Approximations \footnotemark[1] 
}

\author{Hang Cheung\footnotemark[2] \and Jinniao Qiu\footnotemark[2] \and Yang Yang\footnotemark[2]}
\date{}
\footnotetext[1]{This work was partially supported by the National Science and Engineering Research Council of Canada (NSERC). }
\footnotetext[2]{Department of Mathematics \& Statistics, University of Calgary, 2500 University Drive NW, Calgary, AB T2N 1N4, Canada. \textit{E-mail}: \texttt{hang.cheung@ucalgary.ca} (H. Cheung), \texttt{jinniao.qiu@ucalgary.ca} (J. Qiu), \texttt{yang.yang1@ucalgary.ca} (Y. Yang).}

\maketitle

\begin{abstract}
\noindent 
In this paper, we present a novel Feynman-Kac formula and investigate learning-based methods for approximating general nonlinear time-dependent Schr\"odinger equations which may be high-dimensional. Our formulation integrates both the Fisk-Stratonovich and It\^o integrals within the framework of backward stochastic differential equations (BSDEs). Utilizing this Feynman-Kac representation, we propose learning-based approaches for numerical approximations. To demonstrate the accuracy and effectiveness of the proposed method, we conduct numerical experiments in both low- and high-dimensional settings, complemented by a convergence analysis. These results address the open problem concerning deep-BSDE methods for numerical approximations of high-dimensional time-dependent nonlinear Schr\"odinger equations (cf.  [Proc. Natl. Acad. Sci. 15 (2018), pp. 8505–8510] and [Frontiers Sci. Awards Math. (2024), pp. 1-14] by Han, Jentzen, and E).
\end{abstract}



\section{Introduction}
The Schr\"odinger equation  is a fundamental equation in quantum mechanics that describes the temporal evolution of the quantum state of a physical system. It is central to the theory of quantum mechanics and essential for the development of quantum theory, with significant applications in fields such as chemistry, materials science, and quantum computing. In this work, we consider the following Cauchy problem for deterministic backward nonlinear time-dependent Schr\"odinger equation for  the wave function of a quantum mechanical system:
\begin{equation}\label{backward-NSch}
  \begin{array}{l}
    \begin{split}
      i \partial_t u(t,x) & = \frac{\nu}{2}\Delta u (t,x) +  f(t,x,u), \;\; t \leq T;\quad
       u(T,x)  = G(x),
    \end{split}
  \end{array}
\end{equation}
which is equivalent to the classical Schr\"odinger equation via the time-reversing transformation:
$
    u(t,x)\longmapsto
    u(T-t,x), \textup{ for} \; t \leq T,\,\, x\in\bR^d.
$
Here and throughout this work, $\nu>0$ is a real constant and $f(t,x,u)$ is a continuous function with polynomial growth in $u$, for instance, as in \cite{kato1987nonlinear}. The unknown function $u:[0,T]\times\bR^d\rightarrow \bC$ is complex-valued and is written as $u(t,x)=u^R(t,x)+iu^I(t,x)$, where both $u^R$ and $u^I$ are real-valued and defined on $[0,T]\times\bR^d$. Analogously, we write
$$
f(t,x,y)=f^R(t,x,y)+if^I(t,x,y),\quad G(x)=G^R(x)+ i G^I(x), \quad (t,x,y)\in [0,T]\times \bR^d\times \mathbb C,
$$
with $f^R,f^I, G^R,$ and $G^I$ being real-valued continuous functions. Then the nonlinear Schr\"odinger equation above may be written as a coupled system of partial differential equations:
\begin{equation}\label{backward-NSch-1}
  \left\{\begin{array}{l}
    \begin{split}
      -\partial_t u^R(t,x) & = -\frac{\nu}{2}\Delta u^I(t,x) -  f^I(t,x,u), \;\; t \leq T;\quad  u^R(T,x)  = G^R(x);\\
      -\partial_t u^I(t,x) & = \frac{\nu}{2}\Delta u^R(t,x) +  f^R(t,x,u), \;\; t \leq T;\quad \quad u^I(T,x)  = G^I(x).
           \end{split}
  \end{array}\right.
\end{equation}

\noindent Associated to \eqref{backward-NSch-1}, we propose the following backward stochastic differential equation (BSDE):
\begin{equation}\label{BSDE-Sch}
  \left\{\begin{array}{l}
  \begin{split}
  -dY^R_s(t,x)&
            =-\sqrt{\nu} (Z^R_s(t,x)+Z^I_s(t,x))*dW_s\\
            &\quad -f^I(s, \sqrt{\nu}(W_s-W_t)+x, Y_s(t,x))\,ds-\sqrt{\nu}Z^R_s(t,x)\,dW_s;\\
  -dY^I_s(t,x)&
            =\sqrt{\nu} (Z^R_s(t,x)-Z^I_s(t,x))*dW_s
            \\
            &\quad
            	+f^R(s, \sqrt{\nu}(W_s-W_t)+x,Y_s(t,x))\,ds-\sqrt{\nu}Z^I_s(t,x)\,dW_s;
            \\
    Y_T(t,x)&=G( \sqrt{\nu}(W_T-W_t)+x) ,\\
    \end{split}
  \end{array}\right.
\end{equation}
where $Y_s(t,x)=Y^R_s(t,x) + i Y^I_s(t,x)$, $Z_s(t,x)=Z^R_s(t,x) + i Z^I_s(t,x)$, and $*dW_s:=\circ dW_s- dW_s$ with $\circ dW_s$ and $dW_s$ being Fisk-Stratonovich integral and It\^o integral, respectively.
They are related to each other by the following:
\begin{align}\label{reltn-intro-Sch}
 Y_s(t,x)=u(s,\sqrt{\nu}(W_s-W_t)+x),\quad
 Z_s(t,x) = \nabla u(s,\sqrt{\nu}(W_s-W_t)+x),
\end{align} 
for all $(s,x)\in [t,T]\times \mathbb{R}^d$.
The system \eqref{BSDE-Sch} can be viewed as an uncoupled forward-backward stochastic differential equation (FBSDE) but is distinguished by the inclusion of nonstandard integrals $*dW_s$. The term ``uncoupled'' indicates that the backward process does not affect the forward dynamics that is nothing but the Wiener process $(W_t)_{t\geq 0}$, in contrast to ``coupled'' systems where such an interaction is present.  
\\[4pt]
It is now well-established that FBSDEs are intrinsically linked to a large class of parabolic partial differential equations (PDEs) with real-valued coefficients, which may be linear or nonlinear (see, for example, \cite{cheridito2007second,delarue2002existence,Delbaen-Qiu-Tang-NS-2012,HuPengFK95,MaProtterYong1994,ParPeng1992QPDE,PardouxTangFK99}). These links yield generalized Feynman-Kac formulas for the corresponding parabolic PDEs, enabling their resolutions to be reformulated as solving FBSDEs (see \cite{cheridito2007second,lejay2020forward} for instance). Recently, a class of deep learning-based (also called deep-BSDE) methods has been introduced to address the curse of dimensionality often encountered in numerical approximations of such parabolic PDEs; see \cite{beck2019machine,han2017deep,han2018solving,hure2020deep,zhang2004numerical} among others. However, this framework does not extend to time-dependent Schr\"odinger equations such as \eqref{backward-NSch},  due to the presence of the imaginary coefficient before the time derivative. As a result,  existing deep-BSDE methods are not directly applicable in this context.  This limitation along with the broader challenge in numerical approximations of high-dimensional time-dependent nonlinear Schrödinger equations, exists as an open problem in the seminal work \cite{han2018solving} and further discussed in \cite{han2025brief}.  \\
[4pt]
This paper addresses the gap in the availability of a Feynman-Kac representation for nonlinear time-dependent Schrödinger equations. Based on this representation, we propose and investigate learning-based approaches which effectively tackle the curse of dimensionality in numerical simulations. To the best of our knowledge, we are the first to give a Feynman-Kac formula for time-dependent nonlinear Schr\"odinger equations using BSDEs.  This is novel even for the linear case as both complex time transformation and Feynman path integrals (cf. (\cite{albeverio1976mathematical,feynman1948space,kac1949distributions,yan1994feynman})) are avoided in the representation.  Inspired by the deep-BSDE methods introduced in seminal works \cite{beck2019machine,han2017deep,han2018solving,hure2020deep},  we propose a novel learning-based approach for numerical approximations for the Schr\"odinger equation \eqref{backward-NSch}. Our numerical experiments on such BSDEs demonstrate strong agreement with the explicit solutions of the linear and nonlinear Schr\"odinger equations in both low- and high-dimensional settings, as demonstrated in Section \ref{sec_num}. Furthermore, a convergence analysis is provided to support these findings. In the numerical experiments and convergence analysis, we utilize neural network approximations for the unknown functions, enabling the algorithm and convergence analysis to be extended to high-dimensional time-dependent Schrödinger equations; in this context, deep neural networks are instrumental in addressing the curse of dimensionality (cf. \cite{han2018solving,isozaki2004many}).
\\[4pt]
The integrals with $*dW_t$ in BSDE \eqref{BSDE-Sch}, which combine Fisk-Stratonovich and It\^o integrals, capture the leading differential operators, thereby reflecting imaginary coefficient and wave-like behavior of Schr\"odinger equations. This feature poses significant challenges in both  theoretical study and numerical analysis. On the theoretical side, the Feynman-Kac formula is proved for both classical and weak solutions, while several open questions are raised regarding the pure probabilistic resolution of BSDE \eqref{BSDE-Sch} in Section 2.3. From a numerical perspective, this feature distinguishes the proposed learning-based approach (see Algorithm \ref{alg:cap} and Remark \ref{rmk-algo}) from those tailored to non-Schr\"odinger-type parabolic PDEs (cf. \cite{beck2019machine,han2017deep,han2018solving,hure2020deep}). For the challenges in numerical analysis, some non-standard calculations are introduced to derive estimates in the convergence analysis (see Remark \ref{rmk-numerical}). For other numerical methods for nonlinear Schr\"odinger equations, we refer to \cite{bao2024optimal,ji2024low,li2021fully,ostermann2022fully} and references therein.  
\\
[4pt]
This paper is organized as follows: In Section 2, we introduce our Feynman-Kac formula for both classical and weak solutions of general Schrödinger equations  and highlight some open questions arising from this Feynman-Kac representation. In Section 3, we propose a machine learning scheme for numerically computing the BSDEs \eqref{BSDE-Sch}, covering both linear and nonlinear cases of $f$ in both low- and high-dimensional problems. The experimental results demonstrate that the solutions of the BSDEs are in strong agreement with the explicit solutions of the Schrödinger equations. Finally, in Section 4 we present a convergence analysis regarding the proposed algorithm for simulations. Additionally, the appendix contains the estimate of certain error terms involved in the proof of Lemma \ref{Convergence_prelim_result}, along with the complete proof of Lemma \ref{minimization_neural_network}.


\section{Feynman-Kac formula for the nonlinear Schr\"odinger equation}

  A less rigorous form of Feynman-Kac formula is based on the complex time transformation, $\tau=i \, t$. By defining $U(\tau,x)=u(it,x)$, we arrive at the following \textit{parabolic} PDE:
 \begin{equation}\label{backward-NSch-parabolic}
  \begin{array}{l}
    \begin{split}
      - \partial_t U(\tau,x) & = \frac{\nu}{2}\Delta U (\tau,x) +  f(\tau,x,U);\quad
       U(T,x)  = G(x).
    \end{split}
  \end{array}
\end{equation}
This equation may be approached using the established theory of FBSDEs (see  \cite{delarue2002existence,HuPengFK95,MaProtterYong1994,ParPeng1992QPDE,PardouxTangFK99} for instance), which  suggests a representation via a  $\mathbb C$-valued BSDE:
\begin{equation}\label{BSDE-Sch-parabolic}
  \left\{\begin{array}{l}
  \begin{split}
  -d\widetilde Y_s(\tau,x)&
            = f(s, \sqrt{\nu}(W_s-W_{\tau})+x,\widetilde Y_s(\tau,x))\,ds-\sqrt{\nu}\widetilde Z_s(\tau,x)\,dW_s;
            \\
    \widetilde Y_T(\tau,x)&=G( \sqrt{\nu}(W_T-W_{\tau})+x) ,\\
    \end{split}
  \end{array}\right.
\end{equation}
with the relation
\begin{align}\label{reltn-intro-Sch-parabolic}
 \widetilde Y_s(\tau,x)=U(s,\sqrt{\nu}(W_s-W_{\tau})+x),\quad
\widetilde Z_s(t,x) = \nabla U(s,\sqrt{\nu}(W_s-W_{\tau})+x).
\end{align} 
Through equation \eqref{BSDE-Sch-parabolic} and relation \eqref{reltn-intro-Sch-parabolic}, we obtain a Feynman-Kac representation for Schr\"odinger equation \eqref{backward-NSch}. It is important to note that for all equations \eqref{backward-NSch}, \eqref{backward-NSch-parabolic},  \eqref{BSDE-Sch-parabolic}, and \eqref{reltn-intro-Sch-parabolic}, the time domain lies in the \textit{real} interval $[0,T]$. However, the complex time transformation $\tau=i \, t$ necessitates extending the time domain to an imaginary time interval, which imposes significant restrictions on both theory and applications. A formal representation for \textit{linear} Schr\"odinger equations was in fact given via so-called Feynman path integrals. However, these path integrals are not well-defined in most of their uses (cf. \cite{albeverio1976mathematical,feynman1948space,kac1949distributions,yan1994feynman}).
\\
[4pt]
In this section, we shall present a Feynman-Kac formula that avoids both the complex time transformation and Feynman path integrals. The representation is developed for both classical and weak solutions of  Schr\"odinger equation \eqref{BSDE-Sch}. Some further explorations of the resolution of BSDE \eqref{BSDE-Sch} using purely probabilistic methods are left as open questions.

\subsection{Feynman-Kac formula for classical solutions}
Denote by $C^{1,2}([0,T]\times \bR^d; \mathbb C)$ the space of complex-valued continuous functions $u(t,x)$ with bounded continuous derivatives $\partial_t u(t,x)$, $\partial_{x_j} u(t,x)$, and $\partial_{x_kx_l} u(t,x)$, $j,k,l=1,\dots,d$. In the following theorem, the Feynman-Kac formula establishes a connection between the classical solution of Schrödinger equation \eqref{backward-NSch} in $C^{1,2}([0,T]\times \bR^d; \mathbb C)$ and the BSDE \eqref{BSDE-Sch} via the relation \eqref{reltn-intro-Sch}.  

\begin{thm}\label{thm-main}
\begin{enumerate}
\item[(i)] Suppose that $u \in C^{1,2}([0,T]\times \bR^d; \mathbb C)$ is a solution to the Schr\"odinger equation \eqref{backward-NSch}. Then the pair $(Y,Z)$ defined through \eqref{reltn-intro-Sch} is a solution to BSDE \eqref{BSDE-Sch}.
\item[(ii)] If the pair $(Y,Z)$ with $Y_s(t,x)=Y^R_s(t,x) + i Y^I_s(t,x)$ and $Z_s(t,x)=Z^R_s(t,x) + i Z^I_s(t,x)$ for $0\leq t\leq s\leq T, \, x\in\bR^d$ satisfies BSDE \eqref{BSDE-Sch} and there exists a function $u\in C^{1,2}([0,T]\times \bR^d; \mathbb C)$ satisfying the relation \eqref{reltn-intro-Sch}, then the function $u$ is a solution to the Schr\"odinger equation \eqref{backward-NSch}.
\end{enumerate}
\end{thm}
\begin{proof}
First, we prove (i). Without loss of generality, we take $t=0$. Applying It\^o's formula to $u^R$ and $u^I$ respectively yields that 
\begin{align*}
du^R(s,x+\sqrt{\nu}W_s)
&=\left(\partial_s u^R + \frac{\nu}{2} \Delta u^R\right)(s,x+\sqrt{\nu}W_s)\, ds
+\sqrt{\nu} \nabla u^R(s,x+\sqrt{\nu}W_s)\, dW_s,
\\
du^I(s,x+\sqrt{\nu}W_s)
&=\left(\partial_s u^I + \frac{\nu}{2} \Delta u^I\right)(s,x+\sqrt{\nu}W_s)\, ds
+\sqrt{\nu} \nabla u^I(s,x+\sqrt{\nu}W_s)\, dW_s.
\end{align*}
Here, we shall only verify the real part as the imaginary part follows in a similar way. In view of \eqref{backward-NSch} and \eqref{backward-NSch-1}, we have further
\begin{align}
du^R(s,x+\sqrt{\nu}W_s)
&= \left[ \frac{\nu}{2}\Delta( u^I +  u^R) (s,x+\sqrt{\nu}W_s) + f^I(s,x+\sqrt{\nu}W_s,u(s,x+\sqrt{\nu}W_s)) \right]\, ds\nonumber\\
&\quad 
+\sqrt{\nu} \nabla u^R(s,x+\sqrt{\nu}W_s)\, dW_s. \label{eq-thm-R}
\end{align}
Recalling the relationships between Fisk-Stratonovich integrals and It\^o integrals (see \cite[Chapter 4, Exercise 2.18]{revuz2013continuous}), we have
\begin{align}
 \frac{\nu}{2}\Delta( u^I +  u^R) (s,x+\sqrt{\nu}W_s) ds
 &=  \sqrt{\nu}\nabla( u^I +  u^R) (s,x+\sqrt{\nu}W_s)(\circ dW_s-dW_s)
 \nonumber\\
 & = \sqrt{\nu}\nabla( u^I +  u^R) (s,x+\sqrt{\nu}W_s) *dW_s, 
 \label{Stratonovich-ito-relation}
\end{align}
which together with relation \eqref{reltn-intro-Sch} substituted in \eqref{eq-thm-R} gives 
\begin{align*}
  dY^R_s(0,x)&
            =\sqrt{\nu} (Z^R_s(0,x)+Z^I_s(0,x))*dW_s
            +f^I(s, \sqrt{\nu}W_s+x, Y_s(0,x))\,ds+\sqrt{\nu}Z^R_s(0,x)\,dW_s.
\end{align*}
Analogously, we justify the second stochastic differential equation in \eqref{BSDE-Sch} for the imaginary part. The terminal condition is obvious and this justifies (i).\\
[4pt]
For (ii), suppose $(Y,Z)$ and $u$ satisfy BSDE \eqref{BSDE-Sch} as well as the relation \eqref{reltn-intro-Sch}. Let $\varphi$ be an arbitrary infinitely differentiable real-valued function with compact support on $\bR^d$. Then It\^o's formula implies
\begin{align*}
d\varphi(x+\sqrt{\nu}W_s)=\frac{\nu}{2} \Delta \varphi (x+\sqrt{\nu}W_s)\, ds + \sqrt{\nu} \nabla \varphi (x+\sqrt{\nu}W_s)\,dW_s.
\end{align*}
Notice that substituting the representation \eqref{reltn-intro-Sch} into the equation for $Y^R(0,x)$ in  \eqref{BSDE-Sch} gives 
\begin{align*}
du^R(s,x+\sqrt{\nu}W_s)
&= \sqrt{\nu} (\nabla u^R+\nabla u^I)(s,x+\sqrt{\nu}W_s) * dW_s+ f^I(s,x+\sqrt{\nu}W_s,u(s,x+\sqrt{\nu}W_s))  \, ds\nonumber\\
&\quad 
+\sqrt{\nu} \nabla u^R(s,x+\sqrt{\nu}W_s)\, dW_s
\\
&= \left[ \frac{\nu}{2}\Delta( u^I +  u^R) (s,x+\sqrt{\nu}W_s) + f^I(s,x+\sqrt{\nu}W_s,u(s,x+\sqrt{\nu}W_s)) \right]\, ds\nonumber\\
&\quad 
+\sqrt{\nu} \nabla u^R(s,x+\sqrt{\nu}W_s)\, dW_s,
\end{align*}
where we have used the relation \eqref{Stratonovich-ito-relation}. Applying It\^o's formula to the product $(u^R\varphi)(s,x+\sqrt{\nu}W_s) $ and taking integration w.r.t. $(\omega,x)$ on $\Omega\times\bR^d$ under $\bP\otimes dx$ (where $dx$ denotes the Lebesgue measure), we have for each $t\in[0,T]$,
\begin{align}
\int_{\bR^d} u^R(t,x)&\varphi(x)\,dx
=\bE\left[\int_{\bR^d}(u^R\varphi)(t,x+\sqrt{\nu}W_t)\,dx \right]
\notag\\
&=\int_{\bR^d}\bE\left[(u^R\varphi)(t,x+\sqrt{\nu}W_t)\right]\,dx 
\notag\\
&=\int_{\bR^d} \bE\left[ 
(u^R\varphi)(T,x+\sqrt{\nu}W_T)
-\int_t^T \left( \frac{\nu}{2}\Delta( u^I +  u^R)+f^I(\cdot,\cdot,u)\right)\varphi (s,x+\sqrt{\nu}W_s)\,ds\right.
\notag
\\&
\quad
-\int_t^T\nu \left((\nabla u^R)\cdot (\nabla \varphi) + \frac{1}{2}u^R \Delta \varphi  \right)  (s,x+\sqrt{\nu}W_s)\,ds
\notag\\
&\quad 
\left.
-\sqrt{\nu} \int_t^T \left(
\varphi \nabla u^R + u^R \nabla \varphi
\right)(s,x+\sqrt{\nu}W_s)\,dW_s
\right]\,dx
\notag\\
&
=\int_{\bR^d} \bE\left[ 
(u^R\varphi)(T,x+\sqrt{\nu}W_T)
-\int_t^T \left( \frac{\nu}{2}\Delta( u^I +  u^R)+f^I(\cdot,\cdot,u)\right)\varphi (s,x+\sqrt{\nu}W_s)\,ds\right.
\notag
\\&
\quad \left.
-\int_t^T\nu \left((\nabla u^R)\cdot (\nabla \varphi) + \frac{1}{2}u^R \Delta \varphi  \right)  (s,x+\sqrt{\nu}W_s)\,ds 
\right]\,dx
\notag\\
& 
=\bE\Bigg\{\int_{\bR^d} \left[ 
(u^R\varphi)(T,x+\sqrt{\nu}W_T)
-\int_t^T \left( \frac{\nu}{2}\Delta( u^I +  u^R)+f^I(\cdot,\cdot,u)\right)\varphi (s,x+\sqrt{\nu}W_s)\,ds\right.
\notag
\\&
\quad \left.
-\int_t^T\nu \left((\nabla u^R)\cdot (\nabla \varphi) + \frac{1}{2}u^R \Delta \varphi  \right)  (s,x+\sqrt{\nu}W_s)\,ds 
\right]\,dx\Bigg\}
\notag\\
&
=\int_{\bR^d} u^R(T,x)\varphi(x)\,dx
-\int_t^T\int_{\bR^d} \left(\frac{\nu}{2}\Delta u^I(s,x) +f^I(s,x,u(s,x))\right)\varphi(x)\,dx ds,
\label{eq-thm-ii} 
\end{align}
where we have used Fubini's theorem, the fact that 
$\int_{\bR^d} g(x+\sqrt{\nu}W_s) dx=\int_{\bR^d} g(x ) dx,\  \forall\, g\in L^1(\bR^d)$, the integration-by-parts formula
\begin{align*}
\int_{\bR^d} \left[(\nabla u^R)(s,x)\cdot (\nabla \varphi) (x)+ \frac{1}{2}u^R(s,x) \Delta \varphi (x)\right]\,dx
&=\int_{\bR^d} \left[-\Delta u^R(s,x)    \varphi (x)+ \frac{1}{2}\Delta u^R(s,x)  \varphi (x)\right]\,dx
\\
&=-\frac{1}{2}\int_{\bR^d} \Delta u^R(s,x)  \varphi (x) \,dx,
\end{align*}
and the zero-expectation of the stochastic integrals since $\varphi,\, u^R,\, \nabla \varphi$, and $\nabla u^R$ are bounded continuous functions. Thus, by the arbitrariness of the test function $\varphi$, it follows that the function $u^R$ satisfies the partial differential equation
$$-\partial_t u^R(t,x)  = -\frac{\nu}{2}\Delta u^I(t,x) -  f^I(t,x,u);\quad u^R(T,x)=G^R(x).$$ 
In a similar way, we may justify the imaginary part $u^I$ satisfying \eqref{backward-NSch-1} and thus \eqref{backward-NSch}.  
\end{proof}

\subsection{Feynman-Kac formula for weak solutions  }
The associated solution to \eqref{backward-NSch} in Theorem \ref{thm-main} is confined to the space $C^{1,2}([0,T]\times \bR^d; \mathbb C)$ which may be replaced by the set of functions valued in Sobolev spaces, for instance, as in \cite{kato1987nonlinear}. \\ 
[4pt]
For each \((k,q) \in \mathbb{N}_0 \times [1,\infty)\), we define the \(k\)-th Sobolev space \((H^{k,q}, \|\cdot\|_{k,q})\) on \(\mathbb{R}^d\) and its dual space \((H^{-k,q'}, \|\cdot\|_{-k,q'})\), where \(q' = \frac{q}{q-1}\). We denote by \(C^{\infty}_c(\mathbb{R}^d)\) (respectively, \(C^{\infty}_c(\mathcal{O})\) for each open set \(\mathcal{O} \subset \mathbb{R}^d\)) the set of all infinitely differentiable functions with compact support on \(\mathbb{R}^d\) (respectively, \(\mathcal{O}\)). As usual, when $k=0$ we also write $(L^q,\|\cdot\|_q)$ for $(H^{0,q},\|\cdot\|_{0,q})$, and by $\langle \cdot,\,\cdot \rangle$, we denote the duality between the space $L^q$ and its dual $L^{q'}$ for $q\in(1,\infty)$. For simplicity,
if a complex, vector or matrix-valued function $v=(v^{jl})_{1\leq j\leq N,1\leq l\leq n}$ has $v^{jl}\in H^{k,q}$ for $j=1,\dots, N$, $l=1,\dots,n$ with some $n,N\in \bN^+$, we shall write $v\in (H^{k,q})^{N\times n}$ or just simply, $v\in H^{k,q} $ if there is no confusion on dimensions,    with 
$\|v\|_{k,q}=\left( \sum_{l=1}^n\sum_{j=1}^N \big\|v^{jl}\big\|_{k,q}^q \right)^{1/q}$.
Let 
\begin{align}
u\in C^1([0,T]; H^{1,2}) \cap C([0,T]; H^{-1,2})   \text{ and }  f(\cdot,\cdot,u)\in L^p(0,T; L^q )  \text{ for some } p,q\geq1, 
\label{function-space-solution}
\end{align}
 and suppose that such a function $u:\,[0,T]\times \bR^d \rightarrow \mathbb C $ is a weak solution to  \eqref{backward-NSch}, i.e., for any $\varphi\in C^{\infty}_{c}(\bR^d)$, it holds that for each $t\in[0,T]$,
\begin{equation}\label{weak-solution}
\left\{
\begin{split}
\langle u^R(t),\,\varphi\rangle=\langle G^R,\, \varphi\rangle 
+\int_t^T\langle \frac{\nu}{2} \nabla u^I(s),\,\nabla \varphi\rangle ds -\int_t^T\langle f^I(s,u),\,\varphi\rangle\,ds ,
\\
\langle u^I(t),\,\varphi\rangle=\langle G^I,\, \varphi\rangle 
-\int_t^T\langle \frac{\nu}{2} \nabla u^R(s),\,\nabla \varphi\rangle ds + \int_t^T\langle f^R(s,u),\,\varphi \rangle\,ds .
\end{split}
\right.
\end{equation}
Due to  the lack of continuous differentiability, one cannot apply It\^o's formula to the composition $u(t,x+ \sqrt{\nu}W_t)$. Instead, one may apply the It\^o-Kunita-Wentzell-Krylov formula  \cite[Theorem 1]{Krylov_09} to distribution-valued functions and it yields that
\begin{align}
du^R(s,x+\sqrt{\nu}W_s)
&= \left(\frac{\nu}{2}\Delta u^I(s,x+\sqrt{\nu}W_s) +  f^I(s,x+\sqrt{\nu}W_s,u(s,x+\sqrt{\nu}W_s))\right)\,ds
\notag\\
&\quad
+ \frac{\nu}{2} \Delta u^R (s,x+\sqrt{\nu}W_s)\, ds
+\sqrt{\nu} \nabla u^R(s,x+\sqrt{\nu}W_s)\, dW_s
\notag\\
&= \left(  f^I(\cdot,\cdot,u)+ \frac{\nu}{2}(\Delta u^I+\Delta u^R) \right)(s,x+\sqrt{\nu}W_s)\,ds
+\sqrt{\nu} \nabla u^R(s,x+\sqrt{\nu}W_s)\, dW_s,
\notag
\end{align}
 which holds in the distributional sense as in \eqref{weak-solution}.
 \begin{rmk}\label{rmk-divergence}
 For each $g\in L^2(0,T;(L^2)^d)$, the compositions like $g(s,x+\sqrt{\nu}W_s)$ make senses $d\bar{\mathbb{P}}\otimes dt \otimes dx$-a.e. by \cite[Theorem 14.3]{BarlesLesigne_BSDEPDE97_inbook} (see also
\cite[Lemma 3.1]{Delbaen-Qiu-Tang-NS-2012}) and the divergence $\nabla \cdot g(s,x+\sqrt{\nu}W_s)$ may be understood as $H^{-1,2}$-valued process. In view of the  probabilistic interpretation for the divergence (see \cite[Lemma 3.1]{Stoica-2003}), we have the equivalent representation for the integral $*dW_t$, i.e., for each $0\leq t\leq s\leq T$,
\begin{align*}
  \int_t^s\!\!\!g(\tau,x+\sqrt{\nu} W_{\tau}) * dW_{\tau}
  &=\frac{1}{2}\sum_{j=1}^d\left( \int_t^s\!\!\! g^j(\tau,x+\sqrt{\nu}W_{\tau})\,dW_{\tau}^j
  +\!\int_t^s \!\!\!g^j(\tau,x+\sqrt{\nu}W_{\tau})\,\overleftarrow{dW}^j_{\tau}  \right)
  \\
  &=\frac{\sqrt{\nu}}{2}\int_t^s \nabla\cdot g(\tau,x+\sqrt{\nu} W_{\tau})\,d\tau,
  \end{align*}
   with the  integral $\overleftarrow{dW}^j_{\tau}$ being the \emph{backward} stochastic integral (see \cite{NualaPardou98}) and  $dW_{\tau}^j$ the standard It\^o integral.
 \end{rmk}
 \noindent Recalling $u\in C^1([0,T]; H^{1,2})$, we have by Remark \ref{rmk-divergence} that
 \begin{align*}
 du^R(s,x+\sqrt{\nu}W_s) &=   f^I(s,x+\sqrt{\nu}W_s,u(s,x+\sqrt{\nu}W_s))\,ds
 + \sqrt{\nu}(\nabla u^I+\nabla u^R) (s,x+\sqrt{\nu}W_s)* dW_s
 \\
&\quad
+\sqrt{\nu} \nabla u^R(s,x+\sqrt{\nu}W_s)\, dW_s.
 \end{align*}
  By relation \eqref{reltn-intro-Sch}, straightforward substitution gives 
\begin{align*}
  dY^R_s(0,x)&
            =\sqrt{\nu} (Z^R_s(0,x)+Z^I_s(0,x))*dW_s
            +f^I(s, \sqrt{\nu}W_s+x, Y_s(0,x))\,ds+\sqrt{\nu}Z^R_s(0,x)\,dW_s.
\end{align*}
Analogously, we justify the second stochastic differential equation in \eqref{BSDE-Sch} for the imaginary part. \\
[4pt]
On the other hand, the converse holds in a similar way to (ii) of Theorem \ref{thm-main}.\\[4pt]
To sum up, we reach the following result:
\begin{thm}\label{thm-FK-weak-solution}
\begin{enumerate}
\item[(i)] Suppose that $u$ satisfying \eqref{function-space-solution} is a weak solution to the Schr\"odinger equation \eqref{backward-NSch}. Then the pair $(Y_s(t,x),Z_s(t,x))$ defined through \eqref{reltn-intro-Sch} is a solution to BSDE \eqref{BSDE-Sch} for a.e. $x\in\bR^d$.
\item[(ii)] If the pair $(Y,Z)$ with $Y_s(t,x)=Y^R_s(t,x) + i Y^I_s(t,x)$ and $Z_s(t,x)=Z^R_s(t,x) + i Z^I_s(t,x)$ for $0\leq t\leq s\leq T, \, x\in\bR^d$ satisfies BSDE \eqref{BSDE-Sch} for a.e. $x\in\bR^d$ and there exists a function $u $ satisfying \eqref{function-space-solution} and the relation \eqref{reltn-intro-Sch}, then the function $u$ is a solution to the Schr\"odinger equation \eqref{backward-NSch}.
\end{enumerate}
\end{thm}
 \begin{rmk}\label{rmk-thm-weak-solution}
 In the above theorem, a Feynman-Kac formula is established between weak solutions of the Schr\"odinger equation \eqref{backward-NSch} and the associated BSDE \eqref{BSDE-Sch}. The integrability of $f(\cdot,\cdot,u)$ in \eqref{function-space-solution} helps to ensure that the corresponding integrals in \eqref{weak-solution} and BSDE \eqref{BSDE-Sch} are well-defined, and at the same time, it allows $f$ to be nonlinear and of polynomial growth in $u$ (see \cite{kato1987nonlinear}).
 \end{rmk}
\subsection{Open questions}
\label{open_questions}
In Theorems \ref{thm-main} and \ref{thm-FK-weak-solution}, we presuppose the existence of the function $u(t,x)$, either as a classical or weak solution of 
the Schr\"odinger equation \eqref{backward-NSch}, or as a function representing the solution pair $(Y,Z)$ of BSDE \eqref{BSDE-Sch}. In fact, the Schrödinger equation has been extensively studied, as seen in references such as \cite{berezin2012schrodinger,colliander2008global,kato1987nonlinear}, which discuss the existence and uniqueness of solutions. Using the Feynman-Kac formula in Theorems \ref{thm-main} and \ref{thm-FK-weak-solution}, the existence and uniqueness of the solution to the Schrödinger equation \eqref{backward-NSch} implies the unique existence of the solution to BSDE \eqref{BSDE-Sch}. The open question is the converse: without relying on existing well-posedness results of the associated Schrödinger equations, how can we establish the existence and uniqueness of the solution to BSDE \eqref{BSDE-Sch}, from which one may derive the unique existence of solutions to the associated Schrödinger equations?\\
[4pt]
In what follows, we explore BSDE \eqref{BSDE-Sch} within the framework of two established solution theories for BSDEs to better understand the associated challenges.\\
[4pt]
On one hand, in comparison to the standard BSDE theory developed by Pardoux and Peng \cite{ParPeng_90}, the integral terms 
\begin{align}
 \int_{\tau}^T  g(s,Z_s(t,x))*d(\sqrt{\nu}W_s) ,\quad t\leq \tau \leq T,  \label{special-integral}
 \end{align}
need special consideration. Indeed, such integrals are utilized for the interpretation of divergence terms of parabolic equations; see \cite{BarlesLesigne_BSDEPDE97_inbook,Stoica-2003}. The well-posedness of the associated BSDEs involving terms \eqref{special-integral} hinges on the function $g$ being Lipschitz-continuous with respect to $ Z$. Specifically, there exists $\alpha\in (0,1)$ such that
$
|g(s,z_1)-g(s,z_2)| \leq \alpha |z_1-z_2|, \ \forall\, z_1,z_2\in\bR^d,\,s\in[0,T].
$
This condition mandates that the Lipschitz constant $\alpha$ is strictly less than $1$, which is essential for deriving certain estimates of $Z$ in the solution theory; refer to \cite{Stoica-2003} for more details. However, in the context of BSDE \eqref{BSDE-Sch}, we have
$g(s, Z_s(t,x)) = Z^R_s(t,x)\pm  Z^I_s(t,x),\  t\leq \tau \leq s\leq T,$
where the function $g$ remains Lipschitz-continuous in $ Z$, but with a Lipschitz constant $\alpha \geq 1$. This poses a significant challenge in estimating $Z$.\\
[4pt]
On the other hand,  BSDE \eqref{BSDE-Sch} can be interpreted as a $\mathbb C$-valued second order BSDE. In fact, inspired by the work of \cite{cheridito2007second,soner2012wellposedness}, we express $Z$ as an It\^o process satisfying:
$
dZ_s(t,x)= \sqrt{\nu} \Gamma_s(t,x) \,dW_s + A_s(t,x)\,ds, \  t\leq s \leq q.
$
By Remark \ref{rmk-divergence}, we have: for $t\leq \tau\leq T$,
\begin{align}
\int_{\tau}^T\sqrt{\nu} (Z^R_s(t,x)\pm Z^I_s(t,x))*dW_s
&= \int_{\tau}^T \frac{\sqrt{\nu}}{2}\nabla\cdot (\sqrt{\nu}Z^R_s(t,x)\pm \sqrt{\nu}Z^I_s(t,x))\,ds
\notag \\
&
=\int_{\tau}^T \frac{\nu}{2} \text{tr}(\Gamma^R_s(t,x)\pm \Gamma^I_s(t,x))\,ds,
\end{align}
which, substituted into the \eqref{BSDE-Sch}, results in a second order BSDE. In this scenario, $Y_s(t,x)$ is complex-valued, and, to the best of our knowledge, such a BSDE has not been previously studied or addressed in any existing theoretical framework.

\section{Algorithm and Numerical Experiment}
\subsection{An algorithm}
Our BSDE \eqref{BSDE-Sch} can be rewritten as the following  FBSDE: 
\begin{equation}\label{FBSDE-Sch}
  \left\{\begin{array}{l}
  \begin{split}
  d\mathcal{X}_s(t,x) &= \sqrt{\nu}dW_s,\quad
  \mathcal{X}_t(t,x) = x;\\ 
  -dY^R_s(t,x)&
            =-\sqrt{\nu} (Z^R_s(t,x)+Z^I_s(t,x))*dW_s\\
            &\quad -f^I\left(s, \mathcal{X}_s(t,x), Y_s(t,x)\right)\,ds-\sqrt{\nu}Z^R_s(t,x)\,dW_s;\\
  -dY^I_s(t,x)&
            =\sqrt{\nu} (Z^R_s(t,x)-Z^I_s(t,x))*dW_s
            \\
            &\quad
            	+f^R\left(s, \mathcal{X}_s(t,x),Y_s(t,x)\right)\,ds-\sqrt{\nu}Z^I_s(t,x)\,dW_s;
            \\
    Y_T(t,x)&=G(\mathcal{X}_T(t,x)),\\
    \end{split}
  \end{array}\right.
\end{equation}
for all $(s,x)\in [t,T]\times \mathbb{R}^d$. It is related to the solution $u$ of equation \eqref{backward-NSch} by
\begin{align}\label{numeric_relation}
 Y_s(t,x)=u(s,\mathcal{X}_s(t,x)),\quad
 Z_s(t,x) = \nabla u(s,\mathcal{X}_s(t,x)),
\end{align} 
thus the solution $u(s,x)$ would exactly be $Y_s(s,x)$. In our BSDE system \eqref{BSDE-Sch}, the presence of both the Fisk-Stratonovich integral and the Itô integral complicates the numerical computation using any conventional methods (see, for instance, \cite{zhang2004numerical}). Additionally, the deep learning method proposed in \cite{han2018solving} is not capable of handling this atypical form of BSDEs. However, the following will numerically demonstrate the effectiveness of our proposed BSDE approach in solving the Schrödinger equations \eqref{backward-NSch}, using a scheme inspired by \cite{hure2020deep}. \\
[4pt]
For function approximations, in the subsequent presentation of the algorithm, the execution of numerical experiments, and the proof of convergence analysis, we employ neural network approximations for the unknown functions. This approach allows for the extension of both the algorithm and convergence analysis to high-dimensional time-dependent Schr\"odinger equations (cf. \cite{isozaki2004many}); in this context, deep neural networks play a crucial role in tackling the curse of dimensionality (cf. \cite{han2018solving}).\\[4pt]
Let $\mathcal{N}_{d, d_1, L, m}$ denote the set of functions represented by feedforward neural networks with $L+1$ layers (where integer $L>1$), $d$-dimensional inputs, $d_1$-dimensional outputs, and $m$ neurons in each hidden layer. Typically, we represent the parameters for each neural network function by $\theta \in \mathbb{R}^{N_m}$ where $N_m=d(1+m)+m(1+m)(L-2)+m\left(1+d_1\right)$  and $\varrho$ denotes the activation function.\\
[4pt]
For any $M\in\bN^+$, we introduce the time grid $\pi:=\big\{t_0,t_1,\ldots,t_M|0=:t_0<t_1<\ldots<t_M:=T\big\}$, with modulus $|\pi|=\max _{j=0, \ldots, M-1} \Delta t_j, \Delta t_j:=t_{j+1}-t_j$. Without loss of generality, we assume $(t,x)$ to be $(0,x_0)$ for some fixed $x_0 \in \mathbb{R}^d$. We employ the Euler scheme to discretize our forward process in \eqref{FBSDE-Sch} as  
$
X_{t_{j+1}}=X_{t_j}+\sqrt{\nu}\Delta W_{t_j}, \quad j=0, \ldots, M-1, \,X_0=x_0 $,
where we set $\Delta W_{t_j}:=W_{t_{j+1}}-W_{t_j}$. To alleviate notations, we omit the dependence of $X=X^\pi(0,x_0)$ on the time grid $\pi$, initial time $0$ and initial data $x_0$ as there is no ambiguity. Define
\begin{align*}
    &F^R(t,x,u^R,u^I,z^R_{t+1},z^I_{t+1},z^R_{t},z^I_{t},\Delta_t,\Delta_w)\\
    &:= u^R+ \frac{\sqrt{\nu}}{2}(z^R_{t+1} + z^I_{t+1})\Delta_w + \frac{\sqrt{\nu}}{2}(z^R_t - z^I_t)\Delta_w + f^I(t,x,u^R,u^I)\Delta_t,\\
    &F^I(t,x,u^R,u^I,z^R_{t+1},z^I_{t+1},z^R_t,z^I_t,\Delta_t,\Delta_w)\\
    &:=u^I - \frac{\sqrt{\nu}}{2}(z^R_{t+1} - z^I_{t+1})\Delta_w + \frac{\sqrt{\nu}}{2}(z^R_t + z^I_t)\Delta_w - f^R(t,x,u^R,u^I)\Delta_t.
\end{align*}
Here and in what follows, we do not distinguish the writings $f(t,x,u)$ and $f(t,x,u^R,u^I)$. From the equations \eqref{FBSDE-Sch} and \eqref{numeric_relation}, for $j = 1,\ldots, M$, we have that
\begin{align*}
    &u^R(t_{j+1},X_{t_{j+1}})\\
    \approx& F^R\Big(t_j,X_{t_j},u^R(t_j,X_{t_j}),u^I(t_j,X_{t_j}),z^R(t_{j+1},X_{t_{j+1}}),z^I(t_{j+1},X_{t_{j+1}}),z^R(t_j,X_{t_j}),z^I(t_j,X_{t_j}),\Delta t_j, \Delta W_{t_j}\Big);\\
    &u^I(t_{j+1},X_{t_{j+1}})\\
    \approx& F^I\Big(t_j,X_{t_j},u^R(t_j,X_{t_j}),u^I(t_j,X_{t_j}),z^R(t_{j+1},X_{t_{j+1}}),z^I(t_{j+1},X_{t_{j+1}}),z^R(t_j,X_{t_j}),z^I(t_j,X_{t_j}),\Delta t_j, \Delta W_{t_j}\Big).
\end{align*}
Denote $\theta=(\xi,\eta)$. For each $j = 0,\ldots,M-1$, we set up neural networks $\mathcal{U}^R_{t_j}$, $\mathcal{U}^I_{t_j}$, $\mathcal{Z}^R_{t_j}$, $\mathcal{Z}^I_{t_j}\in\mathcal{N}_{d, d_1, L, m}$ to approximate $u^R(t_j,\cdot)$, $u^I(t_j,\cdot)$, $z^R(t_j,\cdot)$, $z^I(t_j,\cdot)$ respectively. The algorithm is stated below.
\begin{algorithm}[H]
\caption{Deep Learning-based scheme for the Schr\"odinger equation \eqref{backward-NSch}}\label{alg:cap}
\begin{algorithmic}
\State \textbf{Initialize: } $\widehat{\mathcal{U}}^R_{t_M} := G^R(\cdot)$, $\widehat{\mathcal{U}}^I_{t_M} := G^I(\cdot)$, $\widehat{\mathcal{Z}}^R_{t_M} := \nabla G^R(\cdot)$, $\widehat{\mathcal{Z}}^I_{t_M} := \nabla G^I(\cdot)$
\For{$j = M-1$ to $0$}
\State Given $\widehat{\mathcal{U}}^R_{t_{j+1}}$, $\widehat{\mathcal{U}}^I_{t_{j+1}}$, $\widehat{\mathcal{Z}}^R_{t_{j+1}}$, $\widehat{\mathcal{Z}}^I_{t_{j+1}}$, minimize 
\begin{align*}
\left\{\begin{aligned}
L_j(\theta)  :&=\mathbb{E}\Big|\widehat{\mathcal{U}}^R_{t_{j+1}}\left(X_{t_{j+1}}\right)-F^R\big(t_j, X_{t_j}, \mathcal{U}^R_{t_j}(X_{t_j} ; \xi), \mathcal{U}^I_{t_j}(X_{t_j} ; \xi),\widehat{\mathcal{Z}}_{t_{j+1}}^R(X_{t_{j+1}}),\widehat{\mathcal{Z}}_{t_{j+1}}^I(X_{t_{j+1}}),\\
&\hspace{20pt}\mathcal{Z}_{t_j}^R(X_{t_j} ; \eta),\mathcal{Z}_{t_j}^I(X_{t_j} ; \eta),\Delta t_j, \Delta W_{t_j}\big)\Big|^2 \\
&\quad
+\mathbb{E}\Big|\widehat{\mathcal{U}}^I_{t_{j+1}}\left(X_{t_{j+1}}\right)-F^I\big(t_j, X_{t_j}, \mathcal{U}^R_{t_j}(X_{t_j} ; \xi), \mathcal{U}^I_{t_j}(X_{t_j} ; \xi),\widehat{\mathcal{Z}}_{t_{j+1}}^R(X_{t_{j+1}}),\widehat{\mathcal{Z}}_{t_{j+1}}^I(X_{t_{j+1}}),\\
&\hspace{220pt}\mathcal{Z}_{t_j}^R(X_{t_j} ; \eta),\mathcal{Z}_{t_j}^I(X_{t_j} ; \eta),\Delta t_j, \Delta W_{t_j}\big)\Big|^2; \\
\theta_j^* &=(\xi^*_j,\eta^*_j) \in \arg \min _{\theta \in \mathbb{R}^{N_m}} L_j(\theta) .
\end{aligned}\right.
\end{align*}
\State Update $\widehat{\mathcal{U}}_{t_j}^R=\mathcal{U}_{t_j}^R\left(. ; \xi_j^*\right)$, $\widehat{\mathcal{U}}_{t_j}^I=\mathcal{U}_{t_j}^I\left(. ; \xi_j^*\right)$, $\widehat{\mathcal{Z}}_{t_j}^R=\mathcal{Z}_{t_j}^R\left(. ; \eta_j^*\right)$, and $\widehat{\mathcal{Z}}_{t_j}^I=\mathcal{Z}_{t_j}^I\left(. ; \eta_j^*\right)$.
\EndFor
\end{algorithmic}
\end{algorithm}
\begin{rmk}\label{rmk-algo}
 The above algorithm is inspired by, yet distinct from, the work of Huré, Pham, and Warin in \cite{hure2020deep}. Indeed, due to the presence of integrals with $*dW_t$ in BSDE \eqref{BSDE-Sch}, at each time step $[t_j,t_{j+1}]$, the iteration functions $F^R$ and $F^I$ depend on the values of
  $Z_{t_{j+1}}$ and $\widehat{\mathcal{Z}}_{t_{j+1}}$ at time $t_{j+1}$, a dependence that is absent in non-Schr\"odinger-type parabolic cases (see, for instance, \cite{beck2019machine,han2017deep,han2018solving,hure2020deep}).
\end{rmk}
\subsection{Numerical experiments}\label{sec_num}
\noindent We conduct numerical experiments in both low-dimensional (1D) and high-dimensional (48D) settings, addressing both linear and nonlinear scenarios. Due to the oscillatory behavior of wave equation solutions, standard activation functions such as Tanh, Sigmoid, and Softplus do not perform well in capturing these intricate patterns. Nonetheless, these activations remain important in less oscillatory regions, where the solution is less curved and rapid variations are absent. To this end, for each \( j = 0, \ldots, M - 1 \), we implement a common neural architecture for the networks \( \mathcal{U}^R_{t_j} \), \( \mathcal{U}^I_{t_j} \), \( \mathcal{Z}^R_{t_j} \), and \( \mathcal{Z}^I_{t_j} \), purposefully designed to accommodate the specific characteristics of wave equation solutions. Each network is constructed using a gated dual-branch framework that integrates two parallel nonlinear transformation paths. One branch applies a sequence of fully connected layers with hyperbolic tangent (Tanh) activations, while the other, inspired by \cite{NEURIPS2020_53c04118}, mirrors this structure but employs sinusoidal activation functions. This configuration equips the network to capture a wide spectrum of input features, from less curved regions to highly oscillatory or periodic behaviors.
\\
[4pt]
Each branch contains two hidden layers: the first linear layer maps the input of dimension \( d \) to a hidden width of either \( d + h \) (in the case of \( \mathcal{U}^R_{t_j} \) and \( \mathcal{U}^I_{t_j} \)) or \( d \times h \) (for \( \mathcal{Z}^R_{t_j} \) and \( \mathcal{Z}^I_{t_j} \)), where \( h \) is a tunable hidden expansion parameter. This is followed by a Softplus activation and a second linear layer of the same width, ending with either a Tanh or sinusoidal activation.\\
[4pt]
The outputs of the two branches are blended using a soft gating mechanism, implemented as a learnable single-neuron feedforward layer followed by a sigmoid activation. This mechanism produces a scalar gating value in \( (0, 1) \) for each input, which is used to compute a convex combination of the two branch outputs. The resulting gated representation is then passed through a final fusion layer, consisting of a single linear transformation. For \( \mathcal{U}^R_{t_j} \) and \( \mathcal{U}^I_{t_j} \), the fusion layer maps the hidden features to a scalar output, while for \( \mathcal{Z}^R_{t_j} \) and \( \mathcal{Z}^I_{t_j} \), it projects to a vector in \( \mathbb{R}^d \). A diagram of the architecture is provided to visually illustrate the computational structure and data flow.
\begin{figure}[H]
    \centering
    \includegraphics[width=0.40\linewidth]{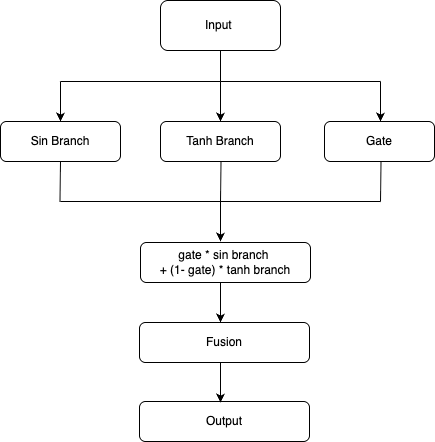}
    \caption{Illustration of the architecture of the proposed neural network.}
    \label{gated_nn}
\end{figure}
\noindent Our code is available at GitHub: \url{https://github.com/HenryCHEUNG7373/Schrodinger_Feynman_Kac}.  
\begin{example}
We consider the following linear Schrodinger equation in $d = 1$:
\begin{equation}\label{Linear_example}
  \begin{array}{l}
    \begin{split}
      i \partial_t u(t,x) & = \frac{1}{2}\Delta u (t,x), \;\; t \leq T;\quad
       u(T,x)  = e^{i\sqrt{2}x},
    \end{split}
  \end{array}
\end{equation}
which admits the solution $u(t,x):= e^{i(\sqrt{2}x-(T-t))}$. We choose $T = 0.5$, $M = 25$, batch size to be $4096$ and train each networks with $30$ epochs. The results are plotted in the following figures.
\begin{figure}[H]
    \centering
    \includegraphics[width=0.49\linewidth]{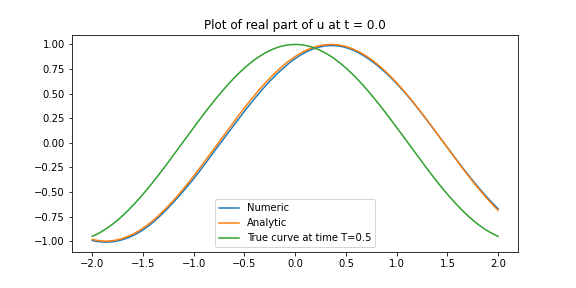}
    \includegraphics[width=0.49\linewidth]{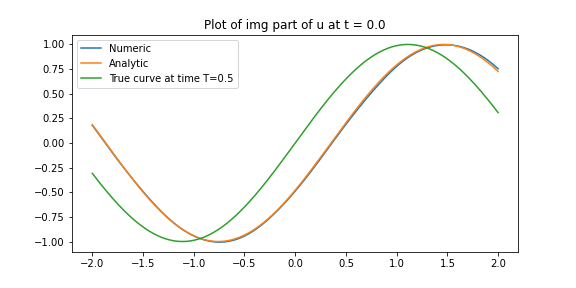}
    \caption{Real \& imaginary parts of $u(0,\cdot)$ and its estimate.}
    \label{ri_0_linear}
\end{figure}

\begin{table}[H]\label{table_l}
\caption{Average value and standard deviation of $u(0,0)$ over 10 independent runs.}
\begin{center}
\begin{tabular}{ | m{5em} | m{5em} | m{12em}| m{12em} | } 
  \hline
  Parts &True value& Average estimated value & Standard deviation \\ 
  \hline
  Real &0.8776 & 0.8630 & 0.0122\\ 
  \hline
  Imaginary &-0.4794 & -0.4728 & 0.0115\\ 
  \hline
\end{tabular}
\end{center}
\end{table}
\end{example}
\begin{example}
\label{nonlinear_example}
Let $d = 1$ and consider the following nonlinear Schrodinger equation:
\begin{equation}\label{Nonlinear_example}
  \begin{array}{l}
    \begin{split}
      i \partial_t u(t,x) & = \frac{1}{2}\Delta u (t,x) + |u(t,x)|^2u(t,x), \;\; t \leq T;\quad
       u(T,x)  = \sech (x)e^{ix},
    \end{split}
  \end{array}
\end{equation}
which has solution $u(t,x):= \sech(x-(T-t))e^{ix}$. Due to high nonlinearity, we take a big $M=64$. For other parameters, we choose $T = 0.5$, batch size to be $4096$ and train each networks with $10$ epochs. The results are plotted in the following figures.
\begin{figure}[H]
    \centering
    \includegraphics[width=0.49\linewidth]{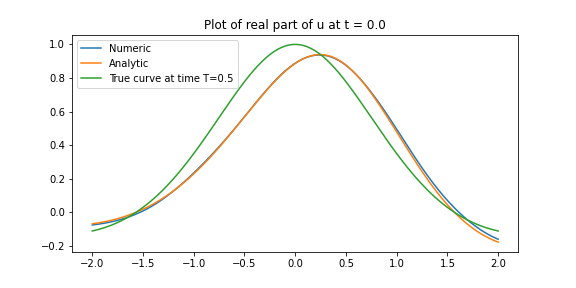}
    \includegraphics[width=0.49\linewidth]{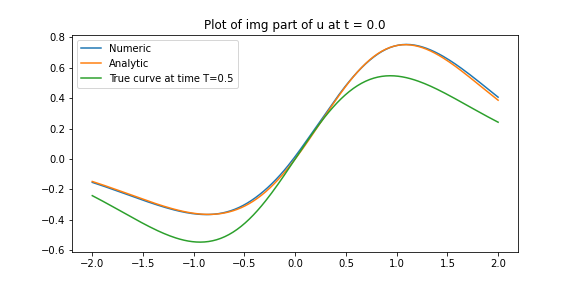}
    \caption{Real \& imaginary parts of $u(0,\cdot)$ and its estimate.}
    \label{r_0_nonlinear}
\end{figure}
\begin{table}[H]\label{table_nl}
\caption{Average value and standard deviation of $u(0,0)$ over 10 independent runs.}
\begin{center}
\begin{tabular}{ | m{5em} | m{5em} | m{12em}| m{12em} | } 
  \hline
  Parts &True value& Average estimated value & Standard deviation \\ 
  \hline
  Real &0.8868 & 0.8723 & 0.0174\\ 
  \hline
  Imaginary &0 & 0.0164 & 0.0142\\ 
  \hline
\end{tabular}
\end{center}
\end{table}
\end{example}
\noindent
The true curves at time $T=0.5$ correspond to the evaluations of the terminal functions, specifically $e^{i\sqrt{2}x}$ in the linear case and $\sech(x)e^{ix}$ in the nonlinear case. The figures above indicate that the computed numerical solution closely approximates the analytical solution at time $t=0$, while both remain noticeably distinct from the prescribed terminal conditions. This is further supported by Tables \ref{table_l} and \ref{table_nl}, where, combining the real and imaginary parts, the relative $L^2$ error, calculated as 
$
\sqrt{\frac{{(\hat{x}_{\text{real}} - x_{\text{real}})^2 + (\hat{x}_{\text{img}} - x_{\text{img}})^2}}{x_{\text{real}}^2 + x_{\text{img}}^2}}
$
for the true value $x = x_{\text{real}} + ix_{\text{img}}$ and the predicted value $\hat{x} =\hat{x}_{\text{real}} + i\hat{x}_{\text{img}}$, is approximately $1.6\%$ in the linear case and around $2.5\%$ in the nonlinear case.\\
[4pt]
In addition to the above one-dimensional examples, we further provide the following linear and nonlinear numerical results  in high-dimensional spaces.
\begin{example}\label{D48_Lin}
    Let $d=48$ and consider the linear Schrödinger equation 
    \begin{equation}\label{D48_Linear_example}
  \begin{array}{l}
    \begin{split}
      i \partial_t u(t,x) & = \frac{1}{2}\Delta u (t,x), \;\; t \leq T;\quad
       u(T,x)  = e^{\sqrt{\frac{2}{d}}i\sum_{j=1}^d x_j}.
    \end{split}
  \end{array}
\end{equation}
    This system admits the exact solution $u(t,x) = e^{i\left(\sqrt{\frac{2}{d}}\sum_{j=1}^d x_j - (T-t)\right)}$. We choose $T=0.5, M=25$, batch size to be 245760 and train each networks with 30 epochs. The results are plotted in the following figures at $t=0$ and $x=(x_1,0,\cdots,0)$.
    \begin{figure}[H]
    \centering
    \includegraphics[width=0.49\linewidth]{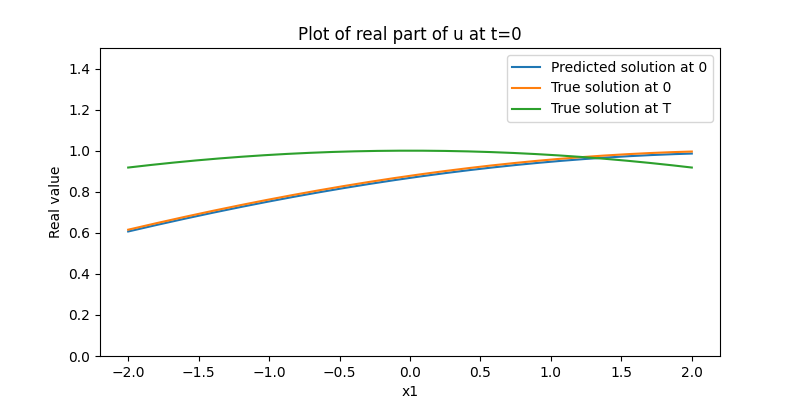}
    \includegraphics[width=0.49\linewidth]{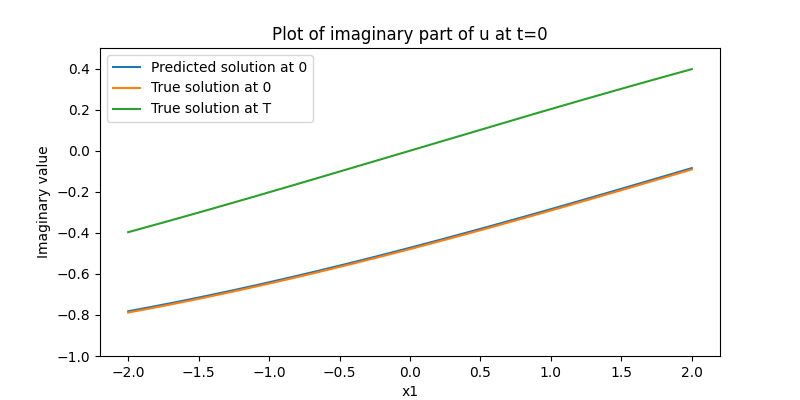}
    \caption{Real \& imaginary parts of $u(0,x)$ and its estimate with $x=(x_1,0,\cdots,0)$.}
    \label{D48_0_linear}
\end{figure}

\begin{table}[H]\label{D48_L}
\begin{center}
\begin{tabular}{|l|l|l|l|}
\hline
Parts     & True value & Average estimated value & Standard deviation \\ \hline
Real      & 0.8776     & 0.8649                  & 0.0074             \\ \hline
Imaginary & -0.4794    & -0.4700                 & 0.0023             \\ \hline
\end{tabular}
\end{center}
\caption{Average value and standard deviation of $u(0,\vec 0)$ over 10 independent runs}
\end{table}

\end{example}

\begin{example}\label{D48_Non_YangSol}
    Let $d=48$ and denote $\mathcal S:=\frac{1}{d}\sum_{j=1}^d x_j$. Consider the nonlinear Schrödinger equation
    \begin{equation}\label{Nonlinear_example_Yang}
  \begin{array}{l}
    \begin{split}
      i \partial_t u(t,x) & = \frac{1}{2}\Delta u (t,x) + f(t,x,u), \;\; t \leq T;\quad
       u(T,x)  = e^{i\mathcal S}\sech (\mathcal S),
    \end{split}
  \end{array}
\end{equation}
    with the nonlinear term 
    $$f(t,x,u) := |u|^2u + \left( 1+\frac{1}{d}\sech^2(\mathcal S) + \frac{i}{d}\tanh(\mathcal S)\right)e^{i\left[T-t+\mathcal S\right]}\sech(\mathcal S) - e^{i\left[T-t+\mathcal S\right]}\sech^3(\mathcal S).$$ 
    
This system admits the exact solution $u(t,x) = e^{i\left(T-t+\mathcal S\right)}\sech(\mathcal S)$. We choose $T=0.5, M=50$, batch size to be 245760 and train each networks with 30 epochs. The results are plotted in the following Figure \ref{D48_0_Non_YangSol} at $t=0$ and $x=(x_1,0,\cdots,0)$. We further provide Figure \ref{D48_0_Non_YangSol_Qiu} at $t=0$ and $x=(x_1,\cdots,x_1)$ to show that our algorithm is capable of approximating solutions with more curving behaviors in high-dimensional nonlinear problems.
\begin{figure}[H]
    \centering
    \includegraphics[width=0.49\linewidth]{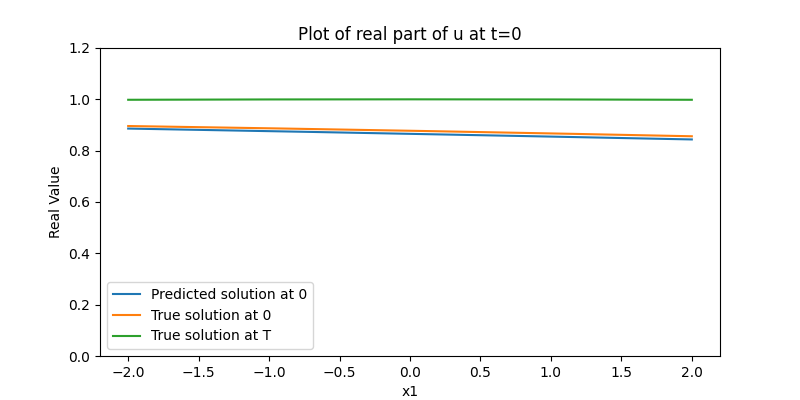}
    \includegraphics[width=0.49\linewidth]{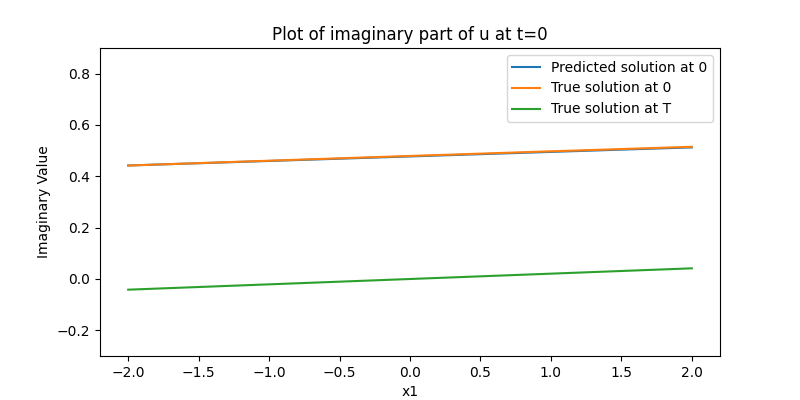}
    \caption{Real \& imaginary parts of $u(0,x)$ and its estimate with $x=(x_1,0,\cdots,0)$.}
    \label{D48_0_Non_YangSol}
\end{figure}
\begin{figure}[H]
    \centering
    \includegraphics[width=0.49\linewidth]{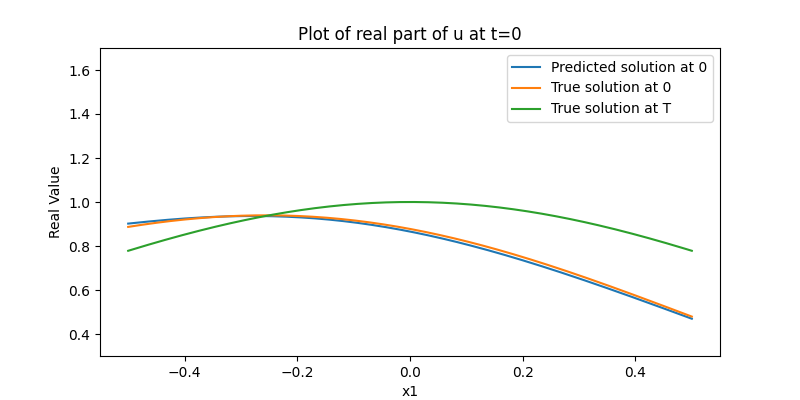}
    \includegraphics[width=0.49\linewidth]{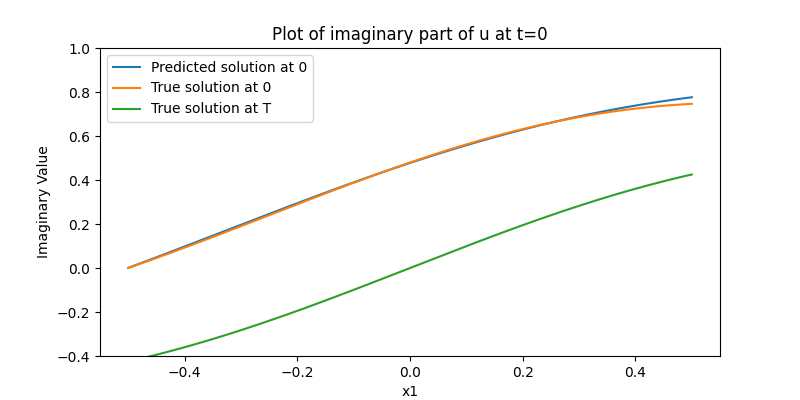}
    \caption{Real \& imaginary parts of $u(0,x)$ and its estimate with $x=(x_1,\cdots,x_1)$.}
    \label{D48_0_Non_YangSol_Qiu}
\end{figure}

\begin{table}[H]\label{D48_Y}
\begin{center}
\begin{tabular}{|l|l|l|l|}
\hline
Parts     & True value & Average estimated value & Standard deviation              \\ \hline
Real      & 0.8776     & 0.8662                  & $8.96\times10^{-4}$ \\ \hline
Imaginary & 0.4794     & 0.4773                  & $4.40\times10^{-4}$ \\ \hline
\end{tabular}
\end{center}
\caption{Average value and standard deviation of $u(0,\vec 0)$ over 10 independent runs.}
\end{table}
\end{example}
\noindent From the above figures, again we observe that our algorithm fits the true solution reasonably well even in the higher dimensional cases with $d=48$. Table \ref{D48_L} and \ref{D48_Y} further supports our numerical results where the relative $L^2$ error can be computed in the same way as in the one-dimensional cases. The linear Example \ref{D48_Lin} exhibits a relative error around $1.58\%$ and in the nonlinear Example \ref{D48_Non_YangSol}, it is approximately $1.16\%$.

\section{Convergence Analysis of Algorithm \ref{alg:cap}}
\subsection{Assumption and notations}
Throughout this section, we assume the following:
\begin{enumerate}
    \item [(i)] The function $u$ is a \textit{bounded} solution to the Schr\"odinger equation \eqref{backward-NSch} with its derivatives $\partial_t u$, $\partial_x u$, $\partial_{xx}u$, $\partial_{xxx}u$, $\partial_{xxxx}u$, $\partial_{t}\partial_{x}u$, $\partial_t\partial_{xx}u$ well defined as continuous functions.
    \item [(ii)] The derivatives of $u$ in (i) above belong to $L^{\infty}([0,T];L^2(\mathbb{R}^d;\mathbb{C}))$ with their norms less than $K$.
    \item [(iii)] There exists $L>0$ such that for all $x,x'\in\bR^d$, $y,y'\in \bC$, $t,t'\in[0,T]$, and $\phi\in C(\bR^d;\mathbb C)$, 
    \begin{align*}
    |f(t,x,y) - f(t,x,y')| & \leq L(|y - y'|),
    \\
    |f(t,x,\phi(x)) - f(t',x',\phi(x))|  & \leq  L \left(\sqrt{\rho(|t-t'|)} +|x-x'|\right) |\phi(x)|,     
\end{align*}
where $\rho:[0,\infty)\rightarrow [0,\infty)$ is a continuous increasing function with $\rho(0)=0$.
\end{enumerate}
    The above conditions (i)--(iii) are assumed for simplicity and to avoid cumbersome arguments. Relaxations of these conditions are feasible.
    Here, we assume the Lipschitz-continuity of $f$ in $y$ in condition (iii). In certain nonlinear cases, such as in  \eqref{Nonlinear_example}, the function $f$ takes the form $f(t,x,u^R,u^I)= |u|^2u$. When confined to bounded solutions as assumed in (i), we may truncate this function by replacing it with $f^N(u)=\Phi^N(|u|^2u)$, where $N>0$ is sufficiently large and $\Phi^N$ is a smooth, compactly supported function that satisfies $\Phi^N(x)=x$ for $|x|\leq N$.   Also, the solution to \eqref{Nonlinear_example} is explicitly given as  $u= \sech(x-(T-t))e^{ix}$ which obviously satisfies both (i) and (ii). Therefore, the example \eqref{Nonlinear_example} actually fits in well with the above assumption.
\\[4pt]
To avoid cumbersome notations, we further assume that $d = 1$, while the multi-dimensional cases follow analogously. Also, w.l.o.g., the time partition is evenly spaced, i.e., $t_{j+1}-t_j = \delta t$ for $j=0,\ldots,M-1$. 
\\[4pt]
Inspired by \cite{hure2020deep}, we first introduce an auxiliary system.  For each $j$, given $\widehat{\mathcal{U}}^R_{t_{j+1}}$, $\widehat{\mathcal{U}}^I_{t_{j+1}}$, $\widehat{\mathcal{Z}}^R_{t_{j+1}}$ and $\widehat{\mathcal{Z}}^I_{t_{j+1}}$, we define the following:
\begin{equation}
\label{true_discrete}
\left\{
\begin{split}
    \widehat{\mathcal{V}}_{t_j}^R := &\mathbb{E}_j\widehat{\mathcal{U}}_{t_{j+1}}^R-\frac{\sqrt{\nu}}{2}\mathbb{E}_j\left[\left(\widehat{\mathcal{Z}}_{t_{j+1}}^R+\widehat{\mathcal{Z}}_{t_{j+1}}^I\right)\delta W_{t_j}\right]-f^I(t_j,X_{t_j},\widehat{\mathcal{V}}_{t_j}^R,\widehat{\mathcal{V}}_{t_j}^I)\delta t,
    \\
    \widehat{\mathcal{V}}_{t_j}^I := &\mathbb{E}_j\widehat{\mathcal{U}}_{t_{j+1}}^I+\frac{\sqrt{\nu}}{2}\mathbb{E}_j\left[\left(\widehat{\mathcal{Z}}_{t_{j+1}}^R-\widehat{\mathcal{Z}}_{t_{j+1}}^I\right)\delta W_{t_j}\right]+f^R(t_j,X_{t_j},\widehat{\mathcal{V}}_{t_j}^R,\widehat{\mathcal{V}}_{t_j}^I) \delta t,
    \\
    \widehat{\mathcal{W}}_{t_j}^R := &\frac{1}{\delta t\sqrt{\nu}}\mathbb{E}_j\left[\left(\widehat{\mathcal{U}}_{t_{j+1}}^R+\widehat{\mathcal{U}}_{t_{j+1}}^I\right)\delta W_{t_j}\right]-\frac{1}{\delta t}\mathbb{E}_j \left[\widehat{\mathcal{Z}}^I_{t_{j+1}}(\delta W_{t_j})^2\right],\\
    \widehat{\mathcal{W}}_{t_j}^I := &\frac{1}{\delta t\sqrt{\nu}}\mathbb{E}_j\left[\left(\widehat{\mathcal{U}}_{t_{j+1}}^I-\widehat{\mathcal{U}}_{t_{j+1}}^R\right)\delta W_{t_j}\right]+\frac{1}{\delta t}\mathbb{E}_j \left[\widehat{\mathcal{Z}}^R_{t_{j+1}}(\delta W_{t_j})^2\right],
\end{split}
\right.
\end{equation}
where $\mathbb{E}_j$ denotes the conditional expectation given $W_{t_1},\dots,W_{t_j}$.
Since $f$ is Lipschitz, by a standard fixed point argument the above system has unique solution for $\delta t$ small enough. From the Markov property of the involved processes, there exist deterministic functions $\hat{v}^R_j$, $\hat{v}^I_j$, $\hat{w}^R_j$ and $\hat{w}^I_j$ such that
\begin{align}
\label{1028}
    \widehat{\mathcal{V}}^R_{t_j} = \hat{v}_j^R(X_{t_j}),\quad \widehat{\mathcal{V}}^I_{t_j} = \hat{v}_j^I(X_{t_j}),\quad\widehat{\mathcal{W}}^R_{t_j} = \hat{w}_j^R(X_{t_j}),\quad\widehat{\mathcal{W}}^I_{t_j} = \hat{w}_j^I(X_{t_j}),\quad j=0,\ldots,M-1. 
\end{align}
Further, by martingale representation theorem there exist integrable processes $\widehat{Z}_t^R$, $\widehat{Z}^I_t$ such that 
{\small
\begin{equation}
\label{martingale_represented_U_V}
\left\{
\begin{split}    \widehat{\mathcal{U}}_{t_{j+1}}^R =& \widehat{\mathcal{V}}_{t_j}^R+ f^I(t_j,X_{t_j},\widehat{\mathcal{V}}_{t_j}^R,\widehat{\mathcal{V}}_{t_j}^I)\delta t+\frac{\sqrt{\nu}}{2}(\widehat{\mathcal{Z}}_{t_{j+1}}^R+\widehat{\mathcal{Z}}_{t_{j+1}}^I)\delta W_{t_j} -\frac{\sqrt{\nu}}{2}(\widehat{\mathcal{W}}_{t_j}^R+\widehat{\mathcal{W}}_{t_j}^I)\delta W_{t_j}+\sqrt{\nu}\int_{t_j}^{t_{j+1}}\widehat{Z}_s^RdW_s,
\\
    \widehat{\mathcal{U}}_{t_{j+1}}^I =& \widehat{\mathcal{V}}_{t_j}^I- f^R(t_j,X_{t_j},\widehat{\mathcal{V}}_{t_j}^R,\widehat{\mathcal{V}}_{t_j}^I)\delta t-\frac{\sqrt{\nu}}{2}(\widehat{\mathcal{Z}}_{t_{j+1}}^R-\widehat{\mathcal{Z}}_{t_{j+1}}^I)\delta W_{t_j} +\frac{\sqrt{\nu}}{2}(\widehat{\mathcal{W}}_{t_j}^R-\widehat{\mathcal{W}}_{t_j}^I)\delta W_{t_j}+\sqrt{\nu}\int_{t_j}^{t_{j+1}}\widehat{Z}_s^IdW_s.
\end{split}
\right.
\end{equation}
}
In this way, it follows straightforwardly that
\begin{align}
     \widehat{\mathcal{W}}_{t_j}^R &= \frac{1}{\delta t}\mathbb{E}_j\int_{t_j}^{t_{j+1}}\widehat{Z}_s^Rds,\quad
    \widehat{\mathcal{W}}_{t_j}^I = \frac{1}{\delta t}\mathbb{E}_j\int_{t_j}^{t_{j+1}}\widehat{Z}_s^Ids.\label{1042}
\end{align}
\\[4pt]
On the other hand, we define the $L^2$-projections of the processes $Z^R$ and $Z^I$ in the true solution to BSDE \eqref{BSDE-Sch}:
\begin{align}
    \overline{Z}_{t_j}^R := \frac{1}{\delta t}\mathbb{E}_j\int_{t_j}^{t_{j+1}}Z_s^Rds,\quad
    \overline{Z}_{t_j}^I := \frac{1}{\delta t}\mathbb{E}_j\int_{t_j}^{t_{j+1}}Z_s^Ids.\label{L2_true_solution_img}
\end{align}
Note that from \eqref{reltn-intro-Sch} and \eqref{1028}, the random variables $Y_{t_j}^R$, $\widehat{\mathcal{V}}_{t_j}^R$, $Z_{t_j}^R$, $\widehat{Z}_{t_j}^R$, $\widehat{\mathcal{W}}_{t_j}^R$ and their imaginary counterparts can be interpreted as deterministic functions, with $X_{t_j}=\mathcal X_{t_j}=x + \sqrt{\nu} W_{t_j}$ as input.
\\[4pt]
Denote $h^R := \nabla u^R(t,x)$, $h^I := \nabla u^I(t,x)$ and recall the relation \eqref{reltn-intro-Sch}, and the forward dynamics $(\mathcal{X}_s)_{0\leq s\leq T}$ and its Euler discretization $X_s$. Define the error terms as follows:
\vspace{0.3cm}
\hrule
\begin{align*}
    \RN{1}_{A,j}^R :=&\frac{\nu}{2}\Big[\partial_x h^R(t_j,\mathcal{X}_{t_j})\delta t-\partial_x h^R(t_j,\mathcal{X}_{t_j})(\delta W_{t_j})^2+\partial_x h^I(t_j,\mathcal{X}_{t_j})\delta t-\partial_x h^I(t_j,\mathcal{X}_{t_j})(\delta W_{t_j})^2\Big]\\
    \RN{1}_{B,j}^R :=& \sqrt{\nu}\int_{t_j}^{t_{j+1}}(Z_s^R+Z_s^I)*dW_s - \frac{\sqrt{\nu}}{2}(Z_{t_{j+1}}^R+Z_{t_{j+1}}^I - Z_{t_j}^R - Z_{t_j}^I)\delta W_{t_j}-\RN{1}_{A,j}^R\\
    \RN{2}_{A,j}^R :=&\sqrt{\nu}\int_{t_j}^{t_{j+1}}Z_s^R dW_s - \sqrt{\nu}\ \overline{Z}_{t_j}^R\delta W_{t_j}\\
    \RN{2}_{B,j}^R :=&\sqrt{\nu}\ \overline{Z}_{t_j}^R\delta W_{t_j} - \sqrt{\nu}Z_{t_j}^R\delta W_{t_j}
    \\\RN{3}_j^R:=&\int_{t_j}^{t_{j+1}}f^I(s,\mathcal X_s,Y_s^{R},Y_s^I)ds - f^I(t_j,\mathcal X_{t_j},Y_{t_j}^R,Y_{t_j}^I)\delta t\\
    \RN{4}_j^R:=&\sqrt{\nu}\int_{t_j}^{t_{j+1}}\widehat{Z}_s^RdW_s - \sqrt{\nu}\widehat{\mathcal{W}}_{t_j}^R\delta W_{t_j}\\
    \RN{1}_{A,j}^I :=&-\frac{\nu}{2}\Big[\Big(\partial_x h^R(t_j,\mathcal{X}_{t_j})\delta t-\partial_x h^R(t_j,\mathcal{X}_{t_j})(\delta W_{t_j})^2\Big)-\Big(\partial_x h^I(t_j,\mathcal{X}_{t_j})\delta t-\partial_x h^I(t_j,\mathcal{X}_{t_j})(\delta W_{t_j})^2\Big)\Big]\\
    \RN{1}_{B,j}^I :=& -\sqrt{\nu}\int_{t_j}^{t_{j+1}}(Z_s^R-Z_s^I)*dW_s + \frac{\sqrt{\nu}}{2}(Z_{t_{j+1}}^R-Z_{t_{j+1}}^I - Z_{t_j}^R + Z_{t_j}^I)\delta W_{t_j}-\RN{1}_{A,j}^I\\
    \RN{2}_{A,j}^I :=&\sqrt{\nu}\int_{t_j}^{t_{j+1}}Z_s^I dW_s - \sqrt{\nu}\ \overline{Z}_{t_j}^I\delta W_{t_j}\\
    \RN{2}_{B,j}^I :=&\sqrt{\nu}\ \overline{Z}_{t_j}^I\delta W_{t_j} - \sqrt{\nu}Z_{t_j}^I\delta W_{t_j}\\
    \RN{3}_j^I:=&-\int_{t_j}^{t_{j+1}}f^R(s,\mathcal X_s,Y_s^{R},Y_s^I)ds + f^R(t_j,\mathcal X_{t_j},Y_{t_j}^R,Y_{t_j}^I)\delta t\\
\RN{4}_j^I:=&\sqrt{\nu}\int_{t_j}^{t_{j+1}}\widehat{Z}_s^IdW_s - \sqrt{\nu}\widehat{\mathcal{W}}_{t_j}^I\delta W_{t_j}.
\end{align*}
\hrule
\vspace{0.4cm}
By \eqref{BSDE-Sch}, for each $j\in\{0,\ldots,M-1\}$ the true solution satisfies:
\begin{equation}
\label{true_Y_solution}
\left\{\begin{split}
    &Y_{t_{j+1}}^R - \frac{\sqrt{\nu}}{2}(Z_{t_{j+1}}^R+Z_{t_{j+1}}^I)\delta W_{t_j} -\RN{1}_{B,j}^R-\RN{2}_{B,j}^R-\RN{3}_j^R -f^I(t_j,\mathcal X_{t_j},Y_{t_j}^R,Y_{t_j}^I)\delta t\\
    =& Y_{t_j}^R - \frac{\sqrt{\nu}}{2}(Z_{t_j}^R+Z_{t_j}^I)\delta W_{t_j}+\sqrt{\nu}Z_{t_j}^R\delta W_{t_j}+\RN{2}_{A,j}^R+\RN{1}_{A,j}^R,\\
    &Y_{t_{j+1}}^I + \frac{\sqrt{\nu}}{2}(Z_{t_{j+1}}^R-Z_{t_{j+1}}^I)\delta W_{t_j} -\RN{1}_{B,j}^I-\RN{2}_{B,j}^I-\RN{3}_j^I +f^R(t_j,\mathcal X_{t_j},Y_{t_j}^R,Y_{t_j}^I)\delta t\\
    =& Y_{t_j}^I + \frac{\sqrt{\nu}}{2}(Z_{t_j}^R-Z_{t_j}^I)\delta W_{t_j}+\sqrt{\nu}Z_{t_j}^I\delta W_{t_j}+\RN{2}_{A,j}^I+\RN{1}_{A,j}^I.
\end{split}\right.
\end{equation}
Meanwhile, \eqref{martingale_represented_U_V} can be represented as  
{\small
\begin{equation}
\label{discrete_U_V}
\left\{\begin{split}
    \widehat{\mathcal{U}}_{t_{j+1}}^R - \frac{\sqrt{\nu}}{2}(\widehat{\mathcal{Z}}_{t_{j+1}}^R+\widehat{\mathcal{Z}}_{t_{j+1}}^I)\delta W_{t_j} -f^I(t_j,X_{t_j},\widehat{\mathcal{V}}_{t_j}^R,\widehat{\mathcal{V}}_{t_j}^I)\delta t=& \widehat{\mathcal{V}}_{t_j}^R - \frac{\sqrt{\nu}}{2}(\widehat{\mathcal{W}}_{t_j}^R+\widehat{\mathcal{W}}_{t_j}^I)\delta W_{t_j}+\sqrt{\nu}\widehat{\mathcal{W}}_{t_j}^R\delta W_{t_j}+\RN{4}_j^R,\\
    \widehat{\mathcal{U}}_{t_{j+1}}^I + \frac{\sqrt{\nu}}{2}(\widehat{\mathcal{Z}}_{t_{j+1}}^R-\widehat{\mathcal{Z}}_{t_{j+1}}^I)\delta W_{t_j}+f^R(t_j,X_{t_j},\widehat{\mathcal{V}}_{t_j}^R,\widehat{\mathcal{V}}_{t_j}^I)\delta t
    =& \widehat{\mathcal{V}}_{t_j}^I + \frac{\sqrt{\nu}}{2}(\widehat{\mathcal{W}}_{t_j}^R-\widehat{\mathcal{W}}_{t_j}^I)\delta W_{t_j}+\sqrt{\nu}\widehat{\mathcal{W}}_{t_j}^I\delta W_{t_j}+\RN{4}_j^I.
\end{split}\right.
\end{equation}}
Combining equations \eqref{true_Y_solution} and \eqref{discrete_U_V} yields
{\small
\begin{equation}
\label{1095}
\left\{
    \begin{split}
        &Y_{t_j}^R - \frac{\sqrt{\nu}}{2}(Z_{t_j}^R+Z_{t_j}^I)\delta W_{t_j}+\sqrt{\nu}Z_{t_j}^R\delta W_{t_j}+\RN{1}_{A,j}^R+\RN{2}_{A,j}^R- \Big[\widehat{\mathcal{V}}_{t_j}^R - \frac{\sqrt{\nu}}{2}(\widehat{\mathcal{W}}_{t_j}^R+\widehat{\mathcal{W}}_{t_j}^I)\delta W_{t_j}+\sqrt{\nu}\widehat{\mathcal{W}}_{t_j}^R\delta W_{t_j}+\RN{4}_j^R\Big]\\
        &
        = Y_{t_{j+1}}^R - \frac{\sqrt{\nu}}{2}(Z_{t_{j+1}}^R+Z_{t_{j+1}}^I)\delta W_{t_j} -\RN{1}_{B,j}^R-\RN{2}_{B,j}^R-\RN{3}_j^R -f^I(t_j,\mathcal X_{t_j},Y_{t_j}^R,Y_{t_j}^I)\delta t 
        \\
        &\quad 
        -\Big[\widehat{\mathcal{U}}_{t_{j+1}}^R - \frac{\sqrt{\nu}}{2}(\widehat{\mathcal{Z}}_{t_{j+1}}^R+\widehat{\mathcal{Z}}_{t_{j+1}}^I)\delta W_{t_j} -f^I(t_j,X_{t_j},\widehat{\mathcal{V}}_{t_j}^R,\widehat{\mathcal{V}}_{t_j}^I)\delta t\Big],
        \\
         &
         Y_{t_j}^I + \frac{\sqrt{\nu}}{2}(Z_{t_j}^R-Z_{t_j}^I)\delta W_{t_j}+\sqrt{\nu}Z_{t_j}^I\delta W_{t_j}+\RN{1}_{A,j}^I+\RN{2}_{A,j}^I - \Big[\widehat{\mathcal{V}}_{t_j}^I + \frac{\sqrt{\nu}}{2}(\widehat{\mathcal{W}}_{t_j}^R-\widehat{\mathcal{W}}_{t_j}^I)\delta W_{t_j}+\sqrt{\nu}\widehat{\mathcal{W}}_{t_j}^I\delta W_{t_j}+\RN{4}_j^I\Big]
         \\
         &
         = Y_{t_{j+1}}^I + \frac{\sqrt{\nu}}{2}(Z_{t_{j+1}}^R-Z_{t_{j+1}}^I)\delta W_{t_j} -\RN{1}_{B,j}^I-\RN{2}_{B,j}^I-\RN{3}_j^I +f^R(t_j,\mathcal X_{t_j},Y_{t_j}^R,Y_{t_j}^I)\delta t
         \\
         &\quad
         -\Big[  \widehat{\mathcal{U}}_{t_{j+1}}^I + \frac{\sqrt{\nu}}{2}(\widehat{\mathcal{Z}}_{t_{j+1}}^R-\widehat{\mathcal{Z}}_{t_{j+1}}^I)\delta W_{t_j}+f^R(t_j,X_{t_j},\widehat{\mathcal{V}}_{t_j}^R,\widehat{\mathcal{V}}_{t_j}^I)\delta t\Big].
    \end{split}
    \right.
\end{equation}
}
\subsection{Convergence analysis}
To simplify notations, by convention we write $\int_{\mathbb R} g$ short for the integral $\int_{\mathbb R}g(x) \, dx$ in this work.  Indeed, integration w.r.t. $x$ helps us obtain estimates which are non-standard in the BSDE literature.
 For $\phi=Y,Z,\widehat{\mathcal{V}}, \widehat{\mathcal{W}}, \widehat{\mathcal{U}}, \widehat{\mathcal{Z}},\widehat{Z}$, recall the notations $\phi=\phi^R + i\phi^I$ and the facts like $\mathbb E\int_{\bR}|\phi^{R/I}(t_j,x+\sqrt{\nu}W_{t_j})|^2=\int_{\bR}|\phi^{R/I}(t_j,x+\sqrt{\nu}W_{t_j})|^2=\int_{\bR}|\phi^{R/I}(t_j,x )|^2$.  
 \\[4pt]
 Define the following errors caused by the neural network approximations:
\begin{align*}
    \varepsilon_j^{v}:=& \int_{\mathbb{R}} \inf_\xi \mathbb{E}|\hat{v}_{t_j}(X_{t_j})-\mathcal{U}_{t_j}(X_{t_j};\xi)|^2,
    \quad\varepsilon_j^{\mathcal W}:= \int_{\mathbb{R}}\inf_\eta  \mathbb{E}|\widehat{\mathcal{W}}_{t_j}(X_{t_j})-\mathcal{Z}_{t_j}(X_{t_j};\eta)|^2.
\end{align*}
For Algorithm \ref{alg:cap}, we have the following convergence analysis whose proof follows straightforwardly by combining Lemmas \ref{minimization_neural_network} and \ref{Convergence_prelim_result}.
\begin{thm}
\label{convergence_Y}
There exists a constant $C>0$ depending on $K$, $L$, $T$, $\nu$ such that the following holds:
    \begin{align}
    \label{1085}
    &\max_{j=0,\ldots,M-1} \int_{\mathbb{R}}\Big[|Y_{t_j}-\widehat{\mathcal{U}}_{t_j}|^2 
    +\frac{(\delta t)\nu}{2} |\overline{Z}_{t_j}-\widehat{\mathcal{Z}}_{t_j}|^2 \Big] 
    \leq 
    \frac{C}{\delta t}\sum_{j=0}^{M-1}\varepsilon_j^{v}+ C\nu\sum_{j=0}^{M-1}\varepsilon_j^{\mathcal W} + C[\delta t+\rho(\delta t)].
\end{align}
\end{thm}
\begin{rmk}
    The right-hand side of \eqref{1085} is the error contributions for Algorithm \ref{alg:cap}, and it consists of four terms. The first two correspond to the error caused by neural network approximations; the better the neural networks are able to approximate or learn the functions $v$ and $w$ for each $j = 0,\cdots,M-1$, the smaller these terms are in the error estimation. From the universal approximation  \cite[theorem $(\RN{1})$]{hornik1989multilayer}, the neural network error terms can be made arbitrarily small for sufficiently large number of neurons and sufficiently many hidden layers. The other two terms  are due to the regularity of $Y$, $Z$ and $f$.
  On the other hand, Theorem \ref{convergence_Y}, mainly addresses the error w.r.t. $Y$ as $\delta t$ tends to zero. This is consisting with Schr\"odinger equation's wave-like feature, which can be illustrated as follows. It\^o formula gives
{\small \begin{align*}
d(Y_s^R)^2 
&= 2 Y_s^R\left[\sqrt{\nu} (Z^R_s+Z^I_s) * dW_s +  f_s^I\, ds + \sqrt{\nu}  Z_s^R\, dW_s\right] + \nu |Z_s^R|^2\, ds,
\\
d(Y^I_s)^2
&=2 Y_s^I\left[-\sqrt{\nu} (Z^R_s-Z^I_s) * dW_s -  f_s^R\, ds + \sqrt{\nu}  Z_s^I\, dW_s\right] + \nu |Z_s^R|^2\, ds.
\end{align*}}
Using integration by parts and the relation \eqref{reltn-intro-Sch}, we obtain
{\small \begin{align*}
\int_{\mathbb{R}} \mathbb{E}|Y_T^R|^2 - \int_{\mathbb{R}} \mathbb{E}|Y_t^R|^2 
&= -\nu \mathbb{E}\int_{\mathbb{R}} \int_t^T Z_s^R Z_s^I\, ds + 2 \mathbb{E} \int_{\mathbb{R}} \int_t^T Y_s^R f_s^I\, ds,
\\
\int_{\mathbb{R}} \mathbb{E}|Y_T^I|^2 - \int_{\mathbb{R}} \mathbb{E}|Y_t^I|^2 
&= \nu \mathbb{E} \int_{\mathbb{R}} \int_t^T Z_s^R Z_s^I\, ds - 2 \mathbb{E} \int_{\mathbb{R}} \int_t^T Y_s^I f_s^R\, ds.
\end{align*}
}Summing these equations leads to the cancellation of the $Z$-terms, highlighting the challenges in deriving non-trivial estimates for $Z_s^R$ and $Z_s^I$.
\end{rmk}
\noindent
For the neural network approximations, we have the following estimate whose proof is more or less standard and postponed to Appendix \ref{appendix:0}.
\begin{lem}
\label{minimization_neural_network}
There exists a constant $C>0$ depending on $L$, $\nu$ such that 
    \begin{align*}
    & \int_{\bR}|\widehat{\mathcal{V}}_{t_j}-\widehat{\mathcal{U}}_{t_j}|^2 
    +\frac{\nu\delta t}{2} \int_{\bR}\Big[|\widehat{\mathcal{W}}_{t_j}-\widehat{\mathcal{Z}}_{t_j}|^2\Big]
    \leq
    C\Big[\varepsilon_j^{v} + (\delta t) \varepsilon_j^{\mathcal W} \Big].
    \end{align*}
\end{lem}
\noindent
The following result measures the distance between $Y_{t_j}$ and $\widehat{\mathcal{U}}_{t_j}$.
\begin{lem}
\label{Convergence_prelim_result}
    There exists a constant $C>0$ depending on $K$, $L$, $T$, $\nu$ such that the following holds:
    \begin{align}
    &\max_{j=0,\ldots,M-1} \int_{\mathbb{R}}\Big[|Y_{t_j}-\widehat{\mathcal{U}}_{t_j}|^2 +\frac{(\delta t)\nu}{2}\Big(|Z_{t_j}-\widehat{\mathcal{Z}}_{t_j}|^2 \Big)\Big] \notag \\
    &\leq 
    \frac{C}{\delta t}\sum_{j=0}^{M-1}\Bigg[ \int_{\mathbb{R}}\Big[|\widehat{\mathcal{U}}_{t_j}-\widehat{\mathcal{V}}_{t_j}|^2 \Bigg]+ C\nu\sum_{j=0}^{M-1}\Bigg[ \int_{\mathbb{R}}\Big(|\widehat{\mathcal{W}}_{t_j}-\widehat{\mathcal{Z}}_{t_j}|^2 \Big)\Big]\Bigg]+C\left[\delta t+\rho(\delta t)\right].\label{lem1-est}
\end{align}
\end{lem}
\begin{proof}
\textbf{Step 1.}   Notice that both sides of \eqref{lem1-est} are deterministic. First, we examine the left-hand side of the real part of \eqref{1095}. Using the $L^2$-projections \eqref{1042} and \eqref{L2_true_solution_img}, we arrive at
\begin{align}
&\mathbb{E}_j\Bigg(\Big(Y_{t_j}^R - \frac{\sqrt{\nu}}{2}(Z_{t_j}^R+Z_{t_j}^I)\delta W_{t_j}+\sqrt{\nu}Z_{t_j}^R\delta W_{t_j} \notag\\
&\hspace{100pt}- \Big[\widehat{\mathcal{V}}_{t_j}^R - \frac{\sqrt{\nu}}{2}(\widehat{\mathcal{W}}_{t_j}^R+\widehat{\mathcal{W}}_{t_j}^I)\delta W_{t_j}+\sqrt{\nu}\widehat{\mathcal{W}}_{t_j}^R\delta W_{t_j}\Big]\Big)+ (\RN{1}_{A,j}^R+\RN{2}_{A,j}^R-\RN{4}_j^R)\Bigg)^2\notag\\
    =&\mathbb{E}_j\left(Y_{t_j}^R - \frac{\sqrt{\nu}}{2}(Z_{t_j}^R+Z_{t_j}^I)\delta W_{t_j}+\sqrt{\nu}Z_{t_j}^R\delta W_{t_j}- \Big[\widehat{\mathcal{V}}_{t_j}^R - \frac{\sqrt{\nu}}{2}(\widehat{\mathcal{W}}_{t_j}^R+\widehat{\mathcal{W}}_{t_j}^I)\delta W_{t_j}+\sqrt{\nu}\widehat{\mathcal{W}}_{t_j}^R\delta W_{t_j}\Big]\right)^2\notag\\
    &+\mathbb{E}_j\left(\RN{1}_{A,j}^R+\RN{2}_{A,j}^R-\RN{4}_j^R\right)^2\notag\\
    \geq&\mathbb{E}_j\left(Y_{t_j}^R - \frac{\sqrt{\nu}}{2}(Z_{t_j}^R+Z_{t_j}^I)\delta W_{t_j}+\sqrt{\nu}Z_{t_j}^R\delta W_{t_j}- \Big[\widehat{\mathcal{V}}_{t_j}^R - \frac{\sqrt{\nu}}{2}(\widehat{\mathcal{W}}_{t_j}^R+\widehat{\mathcal{W}}_{t_j}^I)\delta W_{t_j}+\sqrt{\nu}\widehat{\mathcal{W}}_{t_j}^R\delta W_{t_j}\Big]\right)^2\notag\\
    =&\mathbb{E}_j\Bigg\{\Big[(Y_{t_j}^R-\widehat{\mathcal{V}}_{t_j}^R)+ \Big(-\frac{\sqrt{\nu}}{2}(Z_{t_j}^R+Z_{t_j}^I)\delta W_{t_j}+\frac{\sqrt{\nu}}{2}(\widehat{\mathcal{W}}_{t_j}^R+\widehat{\mathcal{W}}_{t_j}^I)\delta W_{t_j}\Big)\Big]^2\notag\\
&+\nu\Big(\widehat{\mathcal{W}}_{t_j}^R+\widehat{\mathcal{W}}_{t_j}^I - Z_{t_j}^R - Z_{t_j}^I\Big)(\delta W_{t_j})^2(Z_{t_j}^R -\widehat{\mathcal{W}}_{t_j}^R)+\nu(Z_{t_j}^R-\widehat{\mathcal{W}}_{t_j}^R)^2(\delta W_{t_j})^2\Bigg\}.
\label{1109}
\end{align}
Meanwhile, with similar calculation the left-hand side of the imaginary part yields
\begin{align}
    &\mathbb{E}_j\Bigg(Y_{t_j}^I + \frac{\sqrt{\nu}}{2}(Z_{t_j}^R-Z_{t_j}^I)\delta W_{t_j}+\sqrt{\nu}Z_{t_j}^I\delta W_{t_j}\notag \\
    &\hspace{100pt}- \Big[\widehat{\mathcal{V}}_{t_j}^I + \frac{\sqrt{\nu}}{2}(\widehat{\mathcal{W}}_{t_j}^R-\widehat{\mathcal{W}}_{t_j}^I)\delta W_{t_j}+\sqrt{\nu}\widehat{\mathcal{W}}_{t_j}^I\delta W_{t_j}\Big]+\Big(\RN{1}_{A,j}^I+\RN{2}_{A,j}^I-\RN{4}_j^I\Big)\Bigg)^2\notag\\
    &\geq \mathbb{E}_j\Bigg\{\Big[(Y_{t_j}^I-\widehat{\mathcal{V}}_{t_j}^I)+ \Big(\frac{\sqrt{\nu}}{2}(Z_{t_j}^R-Z_{t_j}^I)\delta W_{t_j}-\frac{\sqrt{\nu}}{2}(\widehat{\mathcal{W}}_{t_j}^R-\widehat{\mathcal{W}}_{t_j}^I)\delta W_{t_j}\Big)\Big]^2\notag\\
&\quad +\nu\Big(-\widehat{\mathcal{W}}_{t_j}^R+\widehat{\mathcal{W}}_{t_j}^I + Z_{t_j}^R - Z_{t_j}^I\Big)(\delta W_{t_j})^2(Z_{t_j}^I -\widehat{\mathcal{W}}_{t_j}^I)+\nu(Z_{t_j}^I-\widehat{\mathcal{W}}_{t_j}^I)^2(\delta W_{t_j})^2\Bigg\}.
\label{1119}
\end{align}
Then adding \eqref{1109} and \eqref{1119} leads to
\begin{align*}
    &\mathbb{E}_j\Bigg(\Big(Y_{t_j}^R - \frac{\sqrt{\nu}}{2}(Z_{t_j}^R+Z_{t_j}^I)\delta W_{t_j}+\sqrt{\nu}Z_{t_j}^R\delta W_{t_j}\\
    &\hspace{100pt}- \Big[\widehat{\mathcal{V}}_{t_j}^R - \frac{\sqrt{\nu}}{2}(\widehat{\mathcal{W}}_{t_j}^R+\widehat{\mathcal{W}}_{t_j}^I)\delta W_{t_j}+\sqrt{\nu}\widehat{\mathcal{W}}_{t_j}^R\delta W_{t_j}\Big]\Big)+(\RN{1}_{A,j}^R+\RN{2}_{A,j}^R-\RN{4}_j^R)\Bigg)^2\\
    +&\mathbb{E}_j\Bigg(\Big(Y_{t_j}^I + \frac{\sqrt{\nu}}{2}(Z_{t_j}^R-Z_{t_j}^I)\delta W_{t_j}+\sqrt{\nu}Z_{t_j}^I\delta W_{t_j}\\
    &\hspace{100pt}- \Big[\widehat{\mathcal{V}}_{t_j}^I + \frac{\sqrt{\nu}}{2}(\widehat{\mathcal{W}}_{t_j}^R-\widehat{\mathcal{W}}_{t_j}^I)\delta W_{t_j}+\sqrt{\nu}\widehat{\mathcal{W}}_{t_j}^I\delta W_{t_j}\Big]\Big)+(\RN{1}_{A,j}^I+\RN{2}_{A,j}^I -\RN{4}_j^I)\Bigg)^2\\
    \geq&\mathbb{E}_j\Big[(Y_{t_j}^R-\widehat{\mathcal{V}}_{t_j}^R)+ \Big(-\frac{\sqrt{\nu}}{2}(Z_{t_j}^R+Z_{t_j}^I)\delta W_{t_j}+\frac{\sqrt{\nu}}{2}(\widehat{\mathcal{W}}_{t_j}^R+\widehat{\mathcal{W}}_{t_j}^I)\delta W_{t_j}\Big)\Big]^2\\
    &+\mathbb{E}_j\Big[(Y_{t_j}^I-\widehat{\mathcal{V}}_{t_j}^I)+ \Big(\frac{\sqrt{\nu}}{2}(Z_{t_j}^R-Z_{t_j}^I)\delta W_{t_j}-\frac{\sqrt{\nu}}{2}(\widehat{\mathcal{W}}_{t_j}^R-\widehat{\mathcal{W}}_{t_j}^I)\delta W_{t_j}\Big)\Big]^2.
\end{align*}
Finally, integrating both sides under the Lebesgue measure yields that
\begin{align*}
    &\mathbb{E}_j\int_{\mathbb{R}}\Big[(Y_{t_j}^R-\widehat{\mathcal{V}}_{t_j}^R)+ \Big(-\frac{\sqrt{\nu}}{2}(Z_{t_j}^R+Z_{t_j}^I)\delta W_{t_j}+\frac{\sqrt{\nu}}{2}(\widehat{\mathcal{W}}_{t_j}^R+\widehat{\mathcal{W}}_{t_j}^I)\delta W_{t_j}\Big)\Big]^2\\
    &+\mathbb{E}_j\int_{\mathbb{R}}\Big[(Y_{t_j}^I-\widehat{\mathcal{V}}_{t_j}^I)+ \Big(\frac{\sqrt{\nu}}{2}(Z_{t_j}^R-Z_{t_j}^I)\delta W_{t_j}-\frac{\sqrt{\nu}}{2}(\widehat{\mathcal{W}}_{t_j}^R-\widehat{\mathcal{W}}_{t_j}^I)\delta W_{t_j}\Big)\Big]^2\\
    =&\mathbb{E}_j\int_{\mathbb{R}}\Big[(Y_{t_j}^R-\widehat{\mathcal{V}}_{t_j}^R)^2+ \frac{(\delta t)\nu}{4}\Big((Z_{t_j}^R+Z_{t_j}^I)-(\widehat{\mathcal{W}}_{t_j}^R+\widehat{\mathcal{W}}_{t_j}^I)\Big)^2\Big]\\
    &+\mathbb{E}_j\int_{\mathbb{R}}\Big[(Y_{t_j}^I-\widehat{\mathcal{V}}_{t_j}^I)^2+ \frac{(\delta t)\nu}{4}\Big((Z_{t_j}^R-Z_{t_j}^I)-(\widehat{\mathcal{W}}_{t_j}^R-\widehat{\mathcal{W}}_{t_j}^I)\Big)^2\Big].
\end{align*}
\\[4pt]
\textbf{Step 2.} 
We examine the right-hand side of \eqref{1095}. Using the inequality 
$ (a + b)^2 \leq (1 + \gamma \delta t) a^2 + \left(1 + \frac{1}{\gamma \delta t}\right) b^2 $, 
for some constant $ \gamma > 0 $, we obtain
\begin{align*}
&\mathbb{E}_j\int_{\mathbb{R}} |\textbf{RHS}|^2
\\
&:=\mathbb{E}_j\int_{\mathbb{R}}\Bigg(Y_{t_{j+1}}^R - \frac{\sqrt{\nu}}{2}(Z_{t_{j+1}}^R+Z_{t_{j+1}}^I)\delta W_{t_j} -\RN{1}_{B,j}^R-\RN{2}_{B,j}^R-\RN{3}_j^R - f^I(t_j,\mathcal X_{t_j},Y_{t_j}^R,Y_{t_j}^I)\delta t 
\\
        &\quad
        -\Big[\widehat{\mathcal{U}}_{t_{j+1}}^R - \frac{\sqrt{\nu}}{2}(\widehat{\mathcal{Z}}_{t_{j+1}}^R+\widehat{\mathcal{Z}}_{t_{j+1}}^I)\delta W_{t_j} -f^I(t_j, X_{t_j},\widehat{\mathcal{V}}_{t_j}^R,\widehat{\mathcal{V}}_{t_j}^I)\delta t\Big]\Bigg)^2\\
        &
        \leq (1+\gamma \delta t)\mathbb{E}_j\int_{\mathbb{R}}\Bigg(\Big[Y_{t_{j+1}}^R -\widehat{\mathcal{U}}_{t_{j+1}}^R\Big]+ \frac{\sqrt{\nu}}{2}\delta W_{t_j}\Big[(\widehat{\mathcal{Z}}_{t_{j+1}}^R+\widehat{\mathcal{Z}}_{t_{j+1}}^I)-(Z_{t_{j+1}}^R+Z_{t_{j+1}}^I)\Big]\Bigg)^2 
        \\
        &\quad
        +4\Big(1+\frac{1}{\gamma \delta t}\Big)\mathbb{E}_j\int_{\mathbb{R}}\Big((\RN{1}_{B,j}^R)^2+(\RN{2}_{B,j}^R)^2+(\RN{3}_j^R)^2 +\big(f^I(t_j,\mathcal X_{t_j},Y_{t_j}^R,Y_{t_j}^I)\delta t-f^I(t_j, X_{t_j},\widehat{\mathcal{V}}_{t_j}^R,\widehat{\mathcal{V}}_{t_j}^I)\delta t\big)^2\Big).
        \end{align*}
 Further, in view of the Lipschitz-continuity of function $f$ in condition (iii), we have
        \begin{align*}
        &\mathbb{E}_j\int_{\mathbb{R}} |\textbf{RHS}|^2
\\  
        &
        \leq (1+\gamma \delta t)\mathbb{E}_j\int_{\mathbb{R}}\Bigg(\Big[Y_{t_{j+1}}^R -\widehat{\mathcal{U}}_{t_{j+1}}^R\Big]+ \frac{\sqrt{\nu}}{2}\delta W_{t_j}\Big[(\widehat{\mathcal{Z}}_{t_{j+1}}^R+\widehat{\mathcal{Z}}_{t_{j+1}}^I)-(Z_{t_{j+1}}^R+Z_{t_{j+1}}^I)\Big]\Bigg)^2 \\
        &\quad
        +\frac{8}{\gamma}\Big(\gamma \delta t+1\Big)L^2(\delta t)\Big(\mathbb{E}_j\int_{\mathbb{R}}|Y_{t_j}^R-\widehat{\mathcal{V}}_{t_j}^R|^2+\mathbb{E}_j\int_{\mathbb{R}}|Y_{t_j}^I-\widehat{\mathcal{V}}_{t_j}^I|^2\Big)\Big)\\
        &\quad
        +4\Big(1+\frac{1}{\gamma \delta t}\Big)\mathbb{E}_j\int_{\mathbb{R}}\Big((\RN{1}_{B,j}^R)^2+(\RN{2}_{B,j}^R)^2+(\RN{3}_j^R)^2\Big)\\
        & = (1+\gamma \delta t)\Bigg(\mathbb{E}_j\int_{\mathbb{R}}\Big[Y_{t_{j+1}}^R -\widehat{\mathcal{U}}_{t_{j+1}}^R\Big]^2+ \frac{(\delta t)\nu}{4}\Big[(\widehat{\mathcal{Z}}_{t_{j+1}}^R+\widehat{\mathcal{Z}}_{t_{j+1}}^I)-(Z_{t_{j+1}}^R+Z_{t_{j+1}}^I)\Big]^2\Bigg) \\
        &\quad
        +\frac{8}{\gamma}\Big(\gamma \delta t+1\Big)L^2(\delta t)\Big(\mathbb{E}_j\int_{\mathbb{R}}|Y_{t_j}^R-\widehat{\mathcal{V}}_{t_j}^R|^2+\mathbb{E}_j\int_{\mathbb{R}}|Y_{t_j}^I-\widehat{\mathcal{V}}_{t_j}^I|^2\Big)\Big)\\
        &\quad
        +4\Big(1+\frac{1}{\gamma \delta t}\Big)\mathbb{E}_j\int_{\mathbb{R}}\Big((\RN{1}_{B,j}^R)^2+(\RN{2}_{B,j}^R)^2+(\RN{3}_j^R)^2\Big),
\end{align*}
where in the last equality we have used the facts
\begin{align}
    &\mathbb{E}_j \int_{\mathbb{R}}\Big[Y_{t_{j+1}}^R-\widehat{\mathcal{U}}_{t_{j+1}}^R\Big]\delta W_{t_j}\Big[(\widehat{\mathcal{Z}}_{t_{j+1}}^R+\widehat{\mathcal{Z}}_{t_{j+1}}^I)-(Z_{t_{j+1}}^R+Z_{t_{j+1}}^I)\Big]\notag\\
    &= \mathbb{E}_j\delta W_{t_j}\int_{\mathbb{R}}\Big[Y_{t_{j+1}}^R-\widehat{\mathcal{U}}_{t_{j+1}}^R\Big]\Big[(\widehat{\mathcal{Z}}_{t_{j+1}}^R+\widehat{\mathcal{Z}}_{t_{j+1}}^I)-(Z_{t_{j+1}}^R+Z_{t_{j+1}}^I)\Big]\notag\\
    &=0,\label{demonstrate_leb_x}
\end{align}
and analogously,
\begin{align*}
    &\mathbb{E}_j\int_{\mathbb{R}} (\delta W_{t_j})^2 \Big|(\widehat{\mathcal{Z}}_{t_{j+1}}^R+\widehat{\mathcal{Z}}_{t_{j+1}}^I)-(Z_{t_{j+1}}^R+Z_{t_{j+1}}^I)\Big|^2
    =\delta t  \mathbb{E}_j\int_{\mathbb{R}} 
    \Big|(\widehat{\mathcal{Z}}_{t_{j+1}}^R+\widehat{\mathcal{Z}}_{t_{j+1}}^I)-(Z_{t_{j+1}}^R+Z_{t_{j+1}}^I)\Big|^2.
\end{align*}
\\[4pt]
\textbf{Step 3.} We combine the separate estimates for the left-hand and right-hand sides of \eqref{1095}. Let $\gamma=L+1$. Together with the imaginary part, for $\delta t$ small enough, and a constant $C$ which is possibly changing line by line, we have 
\begin{align*}
&\mathbb{E}_j\int_{\mathbb{R}}\Big[(Y_{t_j}^R-\widehat{\mathcal{V}}_{t_j}^R)^2+ \frac{(\delta t)\nu}{4}\Big((Z_{t_j}^R+Z_{t_j}^I)-(\widehat{\mathcal{W}}_{t_j}^R+\widehat{\mathcal{W}}_{t_j}^I)\Big)^2\Big]\\    &+\mathbb{E}_j\int_{\mathbb{R}}\Big[(Y_{t_j}^I-\widehat{\mathcal{V}}_{t_j}^I)^2+ \frac{(\delta t)\nu}{4}\Big((Z_{t_j}^R-Z_{t_j}^I)-(\widehat{\mathcal{W}}_{t_j}^R-\widehat{\mathcal{W}}_{t_j}^I)\Big)^2\Big]\\
    \leq&(1+C\delta t)\Bigg(\mathbb{E}_j\int_{\mathbb{R}}\Big[Y_{t_{j+1}}^R -\widehat{\mathcal{U}}_{t_{j+1}}^R\Big]^2+ \frac{(\delta t)\nu}{4}\Big[(\widehat{\mathcal{Z}}_{t_{j+1}}^R+\widehat{\mathcal{Z}}_{t_{j+1}}^I)-(Z_{t_{j+1}}^R+Z_{t_{j+1}}^I)\Big]^2\Bigg) \\
    &+(1+C\delta t)\Bigg(\mathbb{E}_j\int_{\mathbb{R}}\Big[Y_{t_{j+1}}^I -\widehat{\mathcal{U}}_{t_{j+1}}^I\Big]^2+ \frac{(\delta t)\nu}{4}\Big[(\widehat{\mathcal{Z}}_{t_{j+1}}^R-\widehat{\mathcal{Z}}_{t_{j+1}}^I)-(Z_{t_{j+1}}^R-Z_{t_{j+1}}^I)\Big]^2\Bigg) \\
        &+8\Big(1+\frac{1}{\gamma \delta t}\Big)\mathbb{E}_j\int_{\mathbb{R}}(\RN{1}_{B,j}^R)^2+(\RN{2}_{B,j}^R)^2+(\RN{3}_j^R)^2+(\RN{1}_{B,j}^I)^2+(\RN{2}_{B,j}^I)^2+(\RN{3}_j^I)^2.
\end{align*}
Using the inequality $(a+b)^2\geq (1-\delta t)a^2 - \frac{1}{\delta t}b^2$, we obtain
\begin{align*}
    &\mathbb{E}_j\int_{\mathbb{R}}\Big[(Y_{t_j}^R-\widehat{\mathcal{V}}_{t_j}^R)^2+ \frac{(\delta t)\nu}{4}\Big((Z_{t_j}^R+Z_{t_j}^I)-(\widehat{\mathcal{W}}_{t_j}^R+\widehat{\mathcal{W}}_{t_j}^I)\Big)^2\Big]\\    &+\mathbb{E}_j\int_{\mathbb{R}}\Big[(Y_{t_j}^I-\widehat{\mathcal{V}}_{t_j}^I)^2+ \frac{(\delta t)\nu}{4}\Big((Z_{t_j}^R-Z_{t_j}^I)-(\widehat{\mathcal{W}}_{t_j}^R-\widehat{\mathcal{W}}_{t_j}^I)\Big)^2\Big]\\
    \geq &(1-\delta t)\Bigg\{\mathbb{E}_j\int_{\mathbb{R}}\Big[(Y_{t_j}^R-\widehat{\mathcal{U}}_{t_j}^R)^2+ \frac{(\delta t)\nu}{4}\Big((Z_{t_j}^R+Z_{t_j}^I)-(\widehat{\mathcal{Z}}_{t_j}^R+\widehat{\mathcal{Z}}_{t_j}^I)\Big)^2\Big]\\    &+\mathbb{E}_j\int_{\mathbb{R}}\Big[(Y_{t_j}^I-\widehat{\mathcal{U}}_{t_j}^I)^2+ \frac{(\delta t)\nu}{4}\Big((Z_{t_j}^R-Z_{t_j}^I)-(\widehat{\mathcal{Z}}_{t_j}^R-\widehat{\mathcal{Z}}_{t_j}^I)\Big)^2\Big]dx\Bigg\}\\
    &-\frac{1}{\delta t}\Bigg\{\mathbb{E}_j\int_{\mathbb{R}}\Big[(\widehat{\mathcal{U}}_{t_j}^R-\widehat{\mathcal{V}}_{t_j}^R)^2+ \frac{(\delta t)\nu}{4}\Big((\widehat{\mathcal{W}}_{t_j}^R+\widehat{\mathcal{W}}_{t_j}^I)-(\widehat{\mathcal{Z}}_{t_j}^R+\widehat{\mathcal{Z}}_{t_j}^I)\Big)^2\Big]\\    &+\mathbb{E}_j\int_{\mathbb{R}}\Big[(\widehat{\mathcal{U}}_{t_j}^I-\widehat{\mathcal{V}}_{t_j}^I)^2+ \frac{(\delta t)\nu}{4}\Big((\widehat{\mathcal{W}}_{t_j}^R-\widehat{\mathcal{W}}_{t_j}^I)-(\widehat{\mathcal{Z}}_{t_j}^R-\widehat{\mathcal{Z}}_{t_j}^I)\Big)^2\Big]\Bigg\}.
\end{align*}
Thus for $\delta t$ small enough, it follows that
\begin{align*}
&\mathbb{E}_j\int_{\mathbb{R}}\Big[(Y_{t_j}^R-\widehat{\mathcal{U}}_{t_j}^R)^2+ \frac{(\delta t)\nu}{4}\Big((Z_{t_j}^R+Z_{t_j}^I)-(\widehat{\mathcal{Z}}_{t_j}^R+\widehat{\mathcal{Z}}_{t_j}^I)\Big)^2\Big]\\    &+\mathbb{E}_j\int_{\mathbb{R}}\Big[(Y_{t_j}^I-\widehat{\mathcal{U}}_{t_j}^I)^2+ \frac{(\delta t)\nu}{4}\Big((Z_{t_j}^R-Z_{t_j}^I)-(\widehat{\mathcal{Z}}_{t_j}^R-\widehat{\mathcal{Z}}_{t_j}^I)\Big)^2\Big]\\
    \leq&(1+C\delta t)\Bigg(\mathbb{E}_j\int_{\mathbb{R}}\Big[Y_{t_{j+1}}^R -\widehat{\mathcal{U}}_{t_{j+1}}^R\Big]^2+ \frac{(\delta t)\nu}{4}\Big[(\widehat{\mathcal{Z}}_{t_{j+1}}^R+\widehat{\mathcal{Z}}_{t_{j+1}}^I)-(Z_{t_{j+1}}^R+Z_{t_{j+1}}^I)\Big]^2\Bigg) \\
    &+(1+C\delta t)\Bigg(\mathbb{E}_j\int_{\mathbb{R}}\Big[Y_{t_{j+1}}^I -\widehat{\mathcal{U}}_{t_{j+1}}^I\Big]^2+ \frac{(\delta t)\nu}{4}\Big[(\widehat{\mathcal{Z}}_{t_{j+1}}^R-\widehat{\mathcal{Z}}_{t_{j+1}}^I)-(Z_{t_{j+1}}^R-Z_{t_{j+1}}^I)\Big]^2\Bigg) \\
        &+C\Big(1+\frac{1}{\gamma \delta t}\Big)\mathbb{E}_j\int_{\mathbb{R}}(\RN{1}_{B,j}^R)^2+(\RN{2}_{B,j}^R)^2+(\RN{3}_j^R)^2+(\RN{1}_{B,j}^I)^2+(\RN{2}_{B,j}^I)^2+(\RN{3}_j^I)^2\\
        &+\frac{C}{\delta t}\Bigg\{\mathbb{E}_j\int_{\mathbb{R}}\Big[(\widehat{\mathcal{U}}_{t_j}^R-\widehat{\mathcal{V}}_{t_j}^R)^2+ \frac{(\delta t)\nu}{4}\Big((\widehat{\mathcal{W}}_{t_j}^R+\widehat{\mathcal{W}}_{t_j}^I)-(\widehat{\mathcal{Z}}_{t_j}^R+\widehat{\mathcal{Z}}_{t_j}^I)\Big)^2\Big]dx\\    &+\mathbb{E}_j\int_{\mathbb{R}}\Big[(\widehat{\mathcal{U}}_{t_j}^I-\widehat{\mathcal{V}}_{t_j}^I)^2+ \frac{(\delta t)\nu}{4}\Big((\widehat{\mathcal{W}}_{t_j}^R-\widehat{\mathcal{W}}_{t_j}^I)-(\widehat{\mathcal{Z}}_{t_j}^R-\widehat{\mathcal{Z}}_{t_j}^I)\Big)^2\Big]\Bigg\}.
\end{align*}
Further simplification gives us 
\begin{align*}
&\mathbb{E}_j\int_{\mathbb{R}}\Big[|Y_{t_j}-\widehat{\mathcal{U}}_{t_j}|^2  +\frac{(\delta t)\nu}{2} |Z_{t_j}-\widehat{\mathcal{Z}}_{t_j}|^2 \Big]\\
\leq&(1+C\delta t)
\Bigg(\mathbb{E}_j\int_{\mathbb{R}}\Big|Y_{t_{j+1}} -\widehat{\mathcal{U}}_{t_{j+1}}\Big|^2
 + \frac{(\delta t)\nu}{2} |\widehat{\mathcal{Z}}_{t_{j+1}}-Z_{t_{j+1}}|^2 \Bigg) 
 +\frac{C}{\delta t}\mathbb{E}_j\int_{\mathbb{R}}\Big[ 
        |\widehat{\mathcal{U}}_{t_j}-\widehat{\mathcal{V}}_{t_j}|^2
        + \frac{(\delta t)\nu}{2} |\widehat{\mathcal{W}}_{t_j}-\widehat{\mathcal{Z}}_{t_j}|^2 \Big]
 \\
        &+16\Big(1+\frac{1}{\gamma \delta t}\Big)\mathbb{E}_j\int_{\mathbb{R}}(\RN{1}_{B,j}^R)^2+(\RN{2}_{B,j}^R)^2+(\RN{3}_j^R)^2+(\RN{1}_{B,j}^I)^2+(\RN{2}_{B,j}^I)^2+(\RN{3}_j^I)^2.
\end{align*}
Recalling $\mathcal{X}_T=X_T$, taking expectations, and using the discrete Gr\"onwall's inequality, we obtain
\begin{align}
    &\max_{j=0,\ldots,M-1}\int_{\mathbb{R}}\Big[|Y_{t_j}-\widehat{\mathcal{U}}_{t_j}|^2
     +\frac{(\delta t)\nu}{2}|Z_{t_j}-\widehat{\mathcal{Z}}_{t_j}|^2 \Big]
     \notag\\ 
    &
    \leq \frac{C}{\delta t}\sum_{j=0}^{M-1}
    \int_{\mathbb{R}}|\widehat{\mathcal{U}}_{t_j}-\widehat{\mathcal{V}}_{t_j}|^2 + C\nu\sum_{j=0}^{M-1}
     \int_{\mathbb{R}}
     |\widehat{\mathcal{W}}_{t_j}-\widehat{\mathcal{Z}}_{t_j}|^2
      \notag\\    
    &\quad
    +\frac{C}{\delta t}\sum_{j=0}^{M-1} 
    \mathbb{E}\int_{\mathbb{R}}
    \left(
|\RN{1}_{B,j}^R|^2+|\RN{2}_{B,j}^R|^2+|\RN{3}_j^R|^2+|\RN{1}_{B,j}^I|^2+|\RN{2}_{B,j}^I|^2+|\RN{3}_j^I|^2\right).
\label{eq-roman-est}
\end{align}
Then the proof may be completed by appropriately estimating the Roman numeral error terms. To avoid cumbersome arguments, the estimate is postponed to Appendix \ref{appendix:1}.
\end{proof}

\begin{rmk}\label{rmk-numerical}
    The proofs of Lemma \ref{Convergence_prelim_result} and Lemma \ref{minimization_neural_network} are inspired by, but distinct from, the work of Huré, Pham, and Warin in \cite{hure2020deep}.   The integrals with $*dW_t$ in BSDE \eqref{BSDE-Sch}, in each time step $[t_j,t_{j+1}]$, generate terms involving the backward processes such as $Z_{t_{j+1}}$, $\widehat{\mathcal{Z}}_{t_{j+1}}$ at time $t_{j+1}$, which are multiplied by the increment of the Brownian motion (see terms of the form $\mathbb{E}_j[(\widehat{\mathcal{Z}}_{t_{j+1}}^R+\widehat{\mathcal{Z}}_{t_{j+1}}^I)\delta W_{t_j}]$ in \eqref{true_discrete} and its counterpart in \eqref{true_Y_solution}). This is in contrast to \cite[Equation 4.5]{hure2020deep}, where no such terms appear. The difficulty arises in the estimation of these terms due to the randomness $x+\sqrt{\nu}W_{t_{j+1}}$ introduced in $Z_{t_{j+1}}$, $\widehat{\mathcal Z}_{t_{j+1}}$ which are not $\sF_{t_j}$-measurable. However, by taking integration over $x$ under the Lebesgue measure, and applying Fubini’s Theorem, we can recover key properties such as $\int_{\mathbb{R}^d}\mathbb{E}_j[(Y_{t_{j+1}}^R-\widehat{\mathcal{U}}_{t_{j+1}}^R)(\widehat{\mathcal{Z}}_{t_{j+1}}^R+\widehat{\mathcal{Z}}_{t_{j+1}}^I)\delta W_{t_j}] = 0$ and $\int_{\mathbb{R}^d}\mathbb{E}_j[(Y_{t_{j+1}}^R-\widehat{\mathcal{U}}_{t_{j+1}}^R)(\widehat{\mathcal{Z}}_{t_{j+1}}^R+\widehat{\mathcal{Z}}_{t_{j+1}}^I)(\delta W_{t_j})^2] 
    = (\delta t)\int_{\mathbb{R}} (Y_{t_{j+1}}^R-\widehat{\mathcal{U}}_{t_{j+1}}^R)(\widehat{\mathcal{Z}}_{t_{j+1}}^R+\widehat{\mathcal{Z}}_{t_{j+1}}^I)$, which are crucial for our analysis. Another significant challenge in our problem arises from the structure of the complex-valued Schr\"odinger equation. This requires us to carefully split certain terms, such as $\RN{1}_{A,j}^R$ and $\RN{1}_{B,j}^R$, allowing us to carry out error calculations through cancellations, measurability arguments, and conditional expectations.
\end{rmk}

\section{Appendix}

\subsection{Proof of Lemma \ref{minimization_neural_network}}
\label{appendix:0}
    \begin{proof}[Proof of Lemma \ref{minimization_neural_network}]
    Recalling equation \eqref{discrete_U_V}, the $L^2$ projection \eqref{1042},  and the real part of cost functional $L_j^R(\theta)$ in Algorithm \ref{alg:cap}, for all neural network parameters $\theta = (\xi,\eta)$ we have
    \begin{align*}
L_j^R(\theta) =& \mathbb{E}\Big|\widehat{\mathcal{V}}_{t_j}^R-\mathcal{U}_{t_j}^R(X_{t_j};\xi)+f^I(t_j,X_{t_j},\widehat{\mathcal{V}}_{t_j}^R,\widehat{\mathcal{V}}_{t_j}^I)\delta t-f^I(t_j,X_{t_j},\mathcal{U}_{t_j}^R(X_{t_j};\xi),\mathcal{U}_{t_j}^I(X_{t_j};\xi))\delta t\\
&+\frac{\sqrt{\nu}}{2}\big(\widehat{\mathcal{W}}_{t_j}^R-\mathcal{Z}_{t_j}^R(X_{t_j};\eta)\big)\delta W_{t_j}-\frac{\sqrt{\nu}}{2}\big(\widehat{\mathcal{W}}_{t_j}^I-\mathcal{Z}_{t_j}^I(X_{t_j};\eta)\big)\delta W_{t_j}\Big|^2 +\nu\mathbb{E}\int_{t_j}^{t_{j+1}}|\widehat{Z}_s^R-\widehat{\mathcal{W}}_{t_j}^R|^2ds\\
=& \mathbb{E}\Big|\widehat{\mathcal{V}}_{t_j}^R-\mathcal{U}_{t_j}^R(X_{t_j};\xi)+f^I(t_j,X_{t_j},\widehat{\mathcal{V}}_{t_j}^R,\widehat{\mathcal{V}}_{t_j}^I)\delta t-f^I(t_j,X_{t_j},\mathcal{U}_{t_j}^R(X_{t_j};\xi),\mathcal{U}_{t_j}^I(X_{t_j};\xi))\delta t\Big|^2\\
&+\mathbb{E}\Big|\frac{\sqrt{\nu}}{2}\big(\widehat{\mathcal{W}}_{t_j}^R-\mathcal{Z}_{t_j}^R(X_{t_j};\eta)\big)\delta W_{t_j}-\frac{\sqrt{\nu}}{2}\big(\widehat{\mathcal{W}}_{t_j}^I-\mathcal{Z}_{t_j}^I(X_{t_j};\eta)\big)\delta W_{t_j}\Big|^2 +\nu\mathbb{E}\int_{t_j}^{t_{j+1}}|\widehat{Z}_s^R-\widehat{\mathcal{W}}_{t_j}^R|^2ds\\
=:& \widetilde{L}_j^R(\theta)+\nu\mathbb{E}\int_{t_j}^{t_{j+1}}|\widehat{Z}_s^R-\widehat{\mathcal{W}}_{t_j}^R|^2ds.
\end{align*}
    Note that the last integral term is a constant independent of the parameters, hence the minimization only involves $\widetilde{L}_j^R$. Similarly we derive the $\widetilde{L}_j^I$ of the imaginary part to be
    \begin{align*}
        \widetilde{L}_j^I(\theta):=&\mathbb{E}\Big|\widehat{\mathcal{V}}_{t_j}^I-\mathcal{U}_{t_j}^I(X_{t_j};\xi)-f^R(t_j,X_{t_j},\widehat{\mathcal{V}}_{t_j}^R,\widehat{\mathcal{V}}_{t_j}^I)\delta t+f^R(t_j,X_{t_j},\mathcal{U}_{t_j}^R(X_{t_j};\xi),\mathcal{U}_{t_j}^I(X_{t_j};\xi))\delta t\Big|^2\\
&+\mathbb{E}\Big|\frac{\sqrt{\nu}}{2}\big(\widehat{\mathcal{W}}_{t_j}^R-\mathcal{Z}_{t_j}^R(X_{t_j};\eta)\big)\delta W_{t_j}+\frac{\sqrt{\nu}}{2}\big(\widehat{\mathcal{W}}_{t_j}^I-\mathcal{Z}_{t_j}^I(X_{t_j};\eta)\big)\delta W_{t_j}\Big|^2.
    \end{align*}
    Adding $\widetilde{L}_j^R$ and $\widetilde{L}_j^I$ together and using the Lipschitz property of $f$ and the parallelogram law yield
    \begin{align*}
        \widetilde{L}_j^R(\theta)+\widetilde{L}_j^I(\theta) \leq& \left(2(1+L\delta t)^2 + 2 L^2(\delta t)^2\right)
        |\widehat{\mathcal{V}}_{t_j}-\mathcal{U}_{t_j}(X_{t_j};\xi)|^2  
        +\frac{\nu(\delta t)}{2}\mathbb{E} |\widehat{\mathcal{W}}_{t_j}-\mathcal{Z}_{t_j}(X_{t_j};\eta)|^2 .
    \end{align*}
    Using the inequality $(a+b)^2\geq (1-\delta t)a^2 - \frac{1}{\delta t}b^2$, we derive that
    \begin{align*}
        \widetilde{L}_j^R(\theta)+\widetilde{L}_j^I(\theta)\geq &(1-C\delta t)\Big[\mathbb{E}|\widehat{\mathcal{V}}_{t_j}-\mathcal{U}_{t_j}(X_{t_j};\xi)|^2 
        \Big]
        +\frac{\nu\delta t}{2}\mathbb{E}\Big[|\widehat{\mathcal{W}}_{t_j}-\mathcal{Z}_{t_j}(X_{t_j};\eta)|^2\Big].
    \end{align*}
    Thus taking $\theta^*_j=(\xi^*_j,\eta^*_j) \in \arg \min _{\theta \in \mathbb{R}^{N_m}} L_j(\theta)$ such that $\widehat{\mathcal{U}}_{t_j}^R = \mathcal{U}_{t_j}^R(X_{t_j};\xi^*_j)$, $\widehat{\mathcal{U}}_{t_j}^I = \mathcal{U}_{t_j}^I(X_{t_j};\xi^*_j)$, $\widehat{\mathcal{Z}}_{t_j}^R = \mathcal{Z}_{t_j}^R(X_{t_j};\eta^*_j)$ and $\widehat{\mathcal{Z}}_{t_j}^I = \mathcal{Z}_{t_j}^I(X_{t_j};\eta^*_j)$,  we have
    \begin{align*}
(1-C\delta t) 
\mathbb{E}|\widehat{\mathcal{V}}_{t_j}-\widehat{\mathcal{U}}_{t_j}|^2   
        &+\frac{\nu\delta t}{2}\mathbb{E} |\widehat{\mathcal{W}}_{t_j}-\widehat{\mathcal{Z}}_{t_j}|^2
        \leq \widetilde{L}_j^R(\theta^*_j)+\widetilde{L}_j^I(\theta^*_j)\leq\widetilde{L}_j^R(\theta)+\widetilde{L}_j^I(\theta) \\
        &\leq \left(2(1+L\delta t)^2 + 2 L^2(\delta t)^2\right)
        \mathbb E\Big[|\widehat{\mathcal{V}}_{t_j}-\mathcal{U}_{t_j}(X_{t_j};\xi)|^2 \Big]
        +
        \frac{\nu(\delta t)}{2}\mathbb{E} |\widehat{\mathcal{W}}_{t_j}-\mathcal{Z}_{t_j}(X_{t_j};\eta)|^2.
    \end{align*}
    Minimizing w.r.t. $\theta$ and integrating w.r.t. $x$ yield the desired estimate when $\delta t$ is small enough.
\end{proof}

\subsection{Estimate of the Roman-numeral error terms in the proof of Lemma \ref{Convergence_prelim_result}}
\label{appendix:1}
\begin{proof}[Complete the proof of Lemma \ref{Convergence_prelim_result}]
We complete the proof of Lemma \ref{Convergence_prelim_result}, by appropriately estimating the error terms associated with Roman numerals in \eqref{eq-roman-est}. We focus on the real parts, as the imaginary components can be treated similarly. First, there holds
\begin{align*}
    \frac{1}{\nu\delta t}\sum_{j=0}^{M-1}\mathbb{E}\int_{\mathbb{R}}|\RN{2}_{B,j}^R|^2
    &
    =
    \sum_{j=0}^{M-1} \frac{1}{\delta t} \int_{\bR} \bE \left| (\overline Z_{t_j}^R)\delta W_{t_j} - (Z_{t_j}^R)\delta W_{t_j} \right|^2 
    \\
    &
    = \sum_{j=0}^{M-1} \frac{1}{(\delta t)^3} \int_{\bR}\! \bE \left[\left( \bE_j \int_{t_j}^{t_{j+1}}\! (Z_s^R-Z_{t_j}^R) ds \right)^2 \delta W_{t_j}^2 \right].
\end{align*}
Recalling $h^R = \nabla u^R(t,x)$ and the relation \eqref{reltn-intro-Sch}, we have by It\^o's formula
{\small
\begin{align*}
    &\sum_{j=0}^{M-1} \frac{1}{(\delta t)^3} \int_{\bR}\! \bE \left[ \left(  \bE_{j}\int_{t_{j}}^{t_{j+1}}\!\! (Z_s^R-Z_{t_{j}}^R) ds \right)^2 \delta W_{t_j}^2 \right] 
    \\
    =& \sum_{j=0}^{M-1} \frac{1}{\delta t^3} \int_{\bR}\! \bE \Bigg[ \Bigg( \bE_j \int_{t_j}^{t_{j+1}}\!\!\Big(\int_{t_{j}}^s\partial_t h^R(k,\mathcal{X}_k)dk
    +\sqrt{\nu}\int_{t_{j}}^s\!\!
    \partial_x h^R(k,\mathcal{X}_k)dW_k
    +\frac{\nu}{2}\int_{t_{j}}^s\!\!\partial_{xx} h^R(k,\mathcal{X}_k)dk\Big)ds \Bigg)^2(\delta W_{t_j})^2\Bigg]
    \\
    =&\sum_{j=0}^{M-1} \frac{1}{\delta t^3} \int_{\bR}\! \bE \Bigg[ \Bigg( \bE_j \int_{t_j}^{t_{j+1}}\!\Big(\int_{t_{j}}^s\partial_t h^R(k,\mathcal{X}_k)dk+\frac{\nu}{2}\int_{t_{j}}^s\partial_{xx} h^R(k,\mathcal{X}_k)dk\Big)ds \Bigg)^2(\delta W_{t_j})^2\Bigg]
    \\
    \leq&\sum_{j=0}^{M-1} \frac{1}{\delta t} \int_{\bR}\! \bE \Bigg[ \Bigg( \bE_j \!\Big(\int_{t_{j}}^{t_{j+1}}|\partial_t h^R(k,\mathcal{X}_k)|dk+\frac{\nu}{2}\int_{t_{j}}^{t_{j+1}}|\partial_{xx} h^R(k,\mathcal{X}_k)|dk\Big) \Bigg)^2(\delta W_{t_j})^2\Bigg]
    \\
    \leq&\sum_{j=0}^{M-1}  \bE \int_{\bR}\! \Bigg[ \Bigg( \bE_j \!\Big(\int_{t_{j}}^{t_{j+1}}2\, |\partial_t h^R(k,\mathcal{X}_k)|^2 dk
    +\frac{\nu^2}{2}\int_{t_{j}}^{t_{j+1}}|\partial_{xx} h^R(k,\mathcal{X}_k)|^2dk\Big) \Bigg)(\delta W_{t_j})^2\Bigg]
    \\
    =&(\delta t)\sum_{j=0}^{M-1}\Bigg(  \!\int_{\mathbb{R}}\int_{t_{j}}^{t_{j+1}}2 |\partial_t h^R(k,x)|^2 dk+\int_{\mathbb{R}}\frac{\nu^2}{2}\int_{t_{j}}^{t_{j+1}}|\partial_{xx} h^R(k,x)|^2dk \Bigg)
    \\
=&(\delta t)\Bigg(  \!\int_{\mathbb{R}}\int_0^T2 |\partial_t h^R(k,x)|^2 dk+\int_{\mathbb{R}}\frac{\nu^2}{2}\int_0^T|\partial_{xx} h^R(k,x)|^2dk \Bigg).
\end{align*}
}
\\[4pt]
Next,   the continuity of function $f$ in condition (iii) implies that 
{\small
\begin{align*}
    \frac{1}{\delta t}\sum_{j=0}^{M-1}\mathbb{E}\int_{\mathbb{R}}(\RN{3}_j^R)^2 
    &
    = \frac{1}{\delta t}\sum_{j=0}^{M-1}\mathbb{E}\int_{\mathbb{R}}\Bigg(\int_{t_j}^{t_{j+1}}f^I(s,\mathcal X_s,Y_s^{R},Y_s^I)ds - f^I(t_j,\mathcal X_{t_j},Y_{t_j}^R,Y_{t_j}^I)\delta t\Bigg)^2 \\
    &\leq \sum_{j=0}^{M-1}\mathbb{E}\int_{\mathbb{R}}\int_{t_j}^{t_{j+1}}|f^I(s,\mathcal X_s,Y_s^{R},Y_s^I) - f^I(t_j,\mathcal X_{t_j},Y_{t_j}^R,Y_{t_j}^I)|^2ds\\
    &\leq  C\sum_{j=0}^{M-1}\mathbb{E}\int_{\mathbb{R}}\int_{t_j}^{t_{j+1}}\left[|Y_s^{R} - Y_{t_j}^R|^2+|Y_s^{I}-Y_{t_j}^I|^2+ (\rho(s-t_{j})+|\mathcal X_s-\mathcal X_{t_j}|^2)|Y_{t_j}|^2\right]ds
    \\
    &\leq  C\left[\rho(\delta t) +\delta t\right] + C\sum_{j=0}^{M-1}\mathbb{E}\int_{\mathbb{R}}\int_{t_j}^{t_{j+1}}\left[|Y_s^{R} - Y_{t_j}^R|^2+|Y_s^{I}-Y_{t_j}^I|^2\right]ds.
\end{align*}} 
We will only look at the real part as the imaginary part is concluded similarly.
\begin{align*}
    &C\sum_{j=0}^{M-1}\mathbb{E}\int_{\mathbb{R}}\int_{t_j}^{t_{j+1}}|Y_s^{R} - Y_{t_j}^R|^2ds
    \\
    &\leq C\sum_{j=0}^{M-1}\mathbb{E}\int_{\mathbb{R}}\int_{t_j}^{t_{j+1}}\left[\Bigg|\int_{t_j}^s \partial_t u(k,\mathcal{X}_k)dk\Bigg|^2 + \Bigg|\int_{t_j}^s \partial_x u(k,\mathcal{X}_k)dW_k\Bigg|^2+\Bigg|\int_{t_j}^s\partial_{xx} u(k,\mathcal{X}_k)dk\Bigg|^2\right]\,ds.
\end{align*}
We concentrate only on the middle term as this term gives us the worst decay, and the other terms are handled similarly. It\^o isometry and straightforward calculations yield that
\begin{align*}
    &C\sum_{j=0}^{M-1}\mathbb{E}\int_{\mathbb{R}}\int_{t_j}^{t_{j+1}}\Bigg|\int_{t_j}^s \partial_x u(k,\mathcal{X}_k)dW_k\Bigg|^2
    \leq (\delta t)C\int_{\mathbb{R}}\int_0^T |\partial_x u(k,\mathcal{X}_k)|^2 dk.
\end{align*}
Thus,
$\frac{1}{\delta t}\sum_{j=0}^{M-1}\mathbb{E}\int_{\mathbb{R}}(\RN{3}_j^R)^2 \leq C \left\{ \delta t + \rho(\delta t)\right\}.$    
\\[4pt]
Now we examine
\begin{align*}
    \frac{1}{\delta t}\sum_{j=0}^{M-1}\mathbb{E}\int_{\mathbb{R}}(\RN{1}_{B,j}^R)^2 =\frac{1}{\delta t}\sum_{j=0}^{M-1} \mathbb{E}\int_{\mathbb{R}}\Bigg(\sqrt{\nu}\int_{t_j}^{t_{j+1}}(Z_s^R+Z_s^I)*dW_s - \frac{\sqrt{\nu}}{2}(Z_{t_{j+1}}^R+Z_{t_{j+1}}^I - Z_{t_j}^R - Z_{t_j}^I)\delta W_{t_j}\\
    -\frac{\nu}{2}\Big[\partial_x h^R(t_j,\mathcal{X}_{t_j})\delta t-\partial_x h^R(t_j,\mathcal{X}_{t_j})(\delta W_{t_j})^2+\partial_x h^I(t_j,\mathcal{X}_{t_j})\delta t-\partial_x h^I(t_j,\mathcal{X}_{t_j})(\delta W_{t_j})^2\Big]\Bigg)^2.
\end{align*}
As the part involving $Z^I$ is similar, we only focus on 
{\small
\begin{align*}
    &\frac{1}{\delta t}\sum_{j=0}^{M-1} \mathbb{E}\int_{\mathbb{R}}\Bigg(\sqrt{\nu}\int_{t_j}^{t_{j+1}}Z_s^R*dW_s - \frac{\sqrt{\nu}}{2}(Z_{t_{j+1}}^R - Z_{t_j}^R)\delta W_{t_j}
    -\frac{\nu}{2}\Big[\partial_x h^R(t_j,\mathcal{X}_{t_j})\delta t-\partial_x h^R(t_j,\mathcal{X}_{t_j})(\delta W_{t_j})^2\Big]\Bigg)^2\\
    \leq&\frac{C}{\delta t}\sum_{j=0}^{M-1} \mathbb{E}\int_{\mathbb{R}}\Bigg[\frac{\nu}{2}\int_{t_j}^{t_{j+1}}\partial_xh^R(s,\mathcal{X}_s)ds - \frac{\nu}{2}\partial_x h^R(t_j,\mathcal{X}_{t_j})\delta t\Bigg]^2\\
    &+\frac{C}{\delta t}\sum_{j=0}^{M-1} \mathbb{E}\int_{\mathbb{R}}\Bigg[\frac{\sqrt{\nu}}{2}(Z_{t_{j+1}}^R - Z_{t_j}^R)\delta W_{t_j} - \frac{\nu}{2}\partial_x h^R(t_j,\mathcal{X}_{t_j})(\delta W_{t_j})^2\Bigg]^2\\
    =:&J_1+J_2.
\end{align*}
}
For $J_1$, applying It\^o's formula yields that
\begin{align*}
    J_1
    =&\frac{C}{\delta t}\sum_{j=0}^{M-1} \mathbb{E}\int_{\mathbb{R}}\Bigg[\frac{\nu}{2}\int_{t_j}^{t_{j+1}}\partial_xh^R(s,\mathcal{X}_s)ds-\frac{\nu}{2}\partial_x h^R(t_j,x+\sqrt{\nu}W_{t_j})\delta t\Bigg]^2\\
    \leq&C\sum_{j=0}^{M-1} \mathbb{E}\int_{\mathbb{R}}\int_{t_j}^{t_{j+1}}\Big|\partial_xh^R(s,\mathcal{X}_s)ds-\partial_x h^R(t_j,x+\sqrt{\nu}W_{t_j})\Big|^2ds\\
    =&C\sum_{j=0}^{M-1} \mathbb{E}\int_{\mathbb{R}}\int_{t_j}^{t_{j+1}}\Big|\int_{t_j}^s\partial_t\partial_x h^R(k,\mathcal{X}_k)dk\Big|^2+\Big|\int_{t_j}^s\partial_{xx} h^R(k,\mathcal{X}_k)dW_k\Big|^2+\Big|\int_{t_j}^s\partial_{xxx} h^R(k,\mathcal{X}_k)dk\Big|^2ds.
\end{align*}
The middle term has the worst decay rate, but it is still enough for our purpose, as 
\begin{align*}
    C\sum_{j=0}^{M-1} \mathbb{E}\int_{\mathbb{R}}\int_{t_j}^{t_{j+1}}\Big|\int_{t_j}^s\partial_{xx} h^R(k,\mathcal{X}_k)dW_k\Big|^2 ds
    &\leq  C\delta t \sum_{j=0}^{M-1} \mathbb{E}\int_{\mathbb{R}} \int_{t_j}^{t_{j+1}}|\partial_{xx} h^R(k,\mathcal{X}_k)|^2dk\\
    &= C\delta t \mathbb{E}\int_{\mathbb{R}}\int_{0}^{T}|\partial_{xx} h^R(k,x)|^2dk.
\end{align*}
For $J_2$, we have
\begin{align*}
    J_2 =& \frac{C}{\delta t}\sum_{j=0}^{M-1} \mathbb{E}\int_{\mathbb{R}}\Bigg[\frac{\sqrt{\nu}}{2}(h^R(t_{j+1},\mathcal{X}_{t_{j+1}}) - h^R(t_{j},\mathcal{X}_{t_j}))\delta W_{t_j} - \frac{\nu}{2}\partial_x h^R(t_j,\mathcal{X}_{t_j})(\delta W_{t_j})^2\Bigg]^2\\
    \leq& \frac{C}{\delta t}\sum_{j=0}^{M-1} \mathbb{E}\int_{\mathbb{R}}\Bigg[\frac{\sqrt{\nu}}{2}(h^R(t_{j+1},\mathcal{X}_{t_{j+1}}) - h^R(t_{j+1},\mathcal{X}_{t_j}))\delta W_{t_j} - \frac{\nu}{2}\partial_x h^R(t_{j+1},\mathcal{X}_{t_j})(\delta W_{t_j})^2\Bigg]^2\\
    &+\frac{C}{\delta t}\sum_{j=0}^{M-1} \mathbb{E}\int_{\mathbb{R}}\Bigg[\frac{\sqrt{\nu}}{2}(h^R(t_{j+1},\mathcal{X}_{t_{j}}) - h^R(t_{j},\mathcal{X}_{t_j}))\delta W_{t_j}\Bigg]^2\\
    &+\frac{C}{\delta t}\sum_{j=0}^{M-1} \mathbb{E}\int_{\mathbb{R}}\Bigg[\frac{\nu}{2}\partial_x h^R(t_{j+1},\mathcal{X}_{t_j})(\delta W_{t_j})^2 - \frac{\nu}{2}\partial_x h^R(t_{j},\mathcal{X}_{t_j})(\delta W_{t_j})^2\Bigg]^2\\
    =:&J_2^1+J_2^2+J_2^3.
\end{align*}
For $J_2^2$, it holds that
\begin{align*}
    J_2^2 &= \frac{C}{\delta t}\sum_{j=0}^{M-1} \mathbb{E}\int_{\mathbb{R}}\Bigg[\Big(h^R(t_{j+1},x+\sqrt{\nu}W_{t_{j}})-h^R(t_{j},x+\sqrt{\nu}W_{t_{j}})\Big)\delta W_{t_j}\Bigg]^2\\
    &
    = \frac{C}{\delta t}\sum_{j=0}^{M-1} \mathbb{E}\int_{\mathbb{R}}\Bigg[\Big(\int_{t_j}^{t_{j+1}}\partial_t h^R(k,x+\sqrt{\nu}W_{t_j})dk\Big)^2(\delta W_{t_j})^2\Bigg]\\
    & 
    \leq C\sum_{j=0}^{M-1} \mathbb{E}\int_{\mathbb{R}}\Bigg[\Big(\int_{t_j}^{t_{j+1}}|\partial_t h^R(k,x+\sqrt{\nu}W_{t_j})|^2dk\Big)(\delta t)\Bigg]\\
    &
    =C\delta t\int_{\mathbb{R}}\int_{0}^{T}|\partial_t h^R(k,x)|^2dk.
\end{align*}
For $J_2^3$, applying Fubini theorem and H\"older's inequality gives
\begin{align*}
    J_2^3 = &\frac{C}{\delta t}\sum_{j=0}^{M-1} \mathbb{E}\int_{\mathbb{R}}\Bigg[\frac{\nu}{2}\partial_x h^R(t_{j+1},\mathcal{X}_{t_j})(\delta W_{t_j})^2 - \frac{\nu}{2}\partial_x h^R(t_{j},\mathcal{X}_{t_j})(\delta W_{t_j})^2\Bigg]^2\\
    =&\frac{C}{\delta t}\sum_{j=0}^{M-1} \mathbb{E}\int_{\mathbb{R}}\Bigg[\int_{t_j}^{t_{j+1}}\partial_t\partial_x h^R(k,x+\sqrt{\nu} W_{t_j})dk\Bigg]^2(\delta W_{t_j})^4\\
    \leq&C\sum_{j=0}^{M-1} \mathbb{E}\int_{\mathbb{R}}\int_{t_j}^{t_{j+1}}|\partial_t\partial_x h^R(k,x+\sqrt{\nu} W_{t_j})|^2dk(\delta W_{t_j})^4\\
    \leq&C|\delta t|^2\int_{\mathbb{R}}\int_0^T|\partial_t\partial_x h^R(k,x)|^2dk.
\end{align*}
For $J_2^1$, simply integrating by parts gives
{\small
\begin{align*}
    J_2^1 
    &
    = \frac{C}{\delta t}\sum_{j=0}^{M-1} \mathbb{E}\int_{\mathbb{R}}\Bigg[\frac{\sqrt{\nu}}{2}(h^R(t_{j+1},\mathcal{X}_{t_{j+1}}) - h^R(t_{j+1},\mathcal{X}_{t_j}))\delta W_{t_j} - \frac{\nu}{2}\partial_x h^R(t_{j+1},\mathcal{X}_{t_j})(\delta W_{t_j})^2\Bigg]^2\\
    &
    = \frac{C}{\delta t}\sum_{j=0}^{M-1}\mathbb{E}\int_{\mathbb{R}}\Bigg[\frac{\sqrt{\nu}}{2}\Bigg(\partial_x h^R(t_{j+1},x+\sqrt{\nu}W_{t_j})\sqrt{\nu}(\delta W_{t_j})^2\\
     &\quad
     +\int_0^1\partial_{xx}h^R\big(t_{j+1},x+\sqrt{\nu}W_{t_j}+\theta(\sqrt{\nu}W_{t_{j+1}}-\sqrt{\nu}W_{t_j})\big)(1-\theta)d\theta\nu(\delta W_{t_j})^3\Bigg)- \frac{\nu}{2}\partial_x h^R(t_{j+1},\mathcal{X}_{t_j})(\delta W_{t_j})^2\Bigg]^2\\
     &
     \leq  \frac{C}{\delta t}\sum_{j=0}^{M-1} \mathbb{E}\int_{\mathbb{R}}\int_0^1|\partial_{xx}h^R\big(t_{j+1},x+\sqrt{\nu}W_{t_j}+\theta(\sqrt{\nu}W_{t_{j+1}}-\sqrt{\nu}W_{t_j})\big)|^2d\theta(\delta W_{t_j})^6\\
    &
    \leq  C|\delta t|^2 \int_{\mathbb{R}}\sum_{j=0}^{M-1}|\partial_{xx}h^R(t_{j+1},x)|^2.
\end{align*}}
Summing up the above estimates yields the desired result.
\end{proof}

\bibliographystyle{siam}
\bibliography{ref_qjn}
\end{document}